\documentclass[11pt]{amsart}
\usepackage{mathtools}       
\usepackage[utf8]{inputenc}  
\usepackage{url}             
\usepackage{amssymb,latexsym}
\usepackage{bbm} 
\usepackage[all]{xy}
\usepackage{color}

\newtheorem{Thm}{Theorem}[section]
\newtheorem{Lem}[Thm]{Lemma}
\newtheorem{Cor}[Thm]{Corollary}
\newtheorem{Prop}[Thm]{Proposition}
\newtheorem{Conj}[Thm]{Conjecture}

\theoremstyle{definition}

\newcommand{\C}{\mathbb{C}}

\newcommand{\N}{\mathbb{N}}
\newcommand{\PP}{\mathbb{P}} 

\newcommand{\Z}{\mathbb{Z}}

\newcommand{\NN}{\mathbb{N}} 

\newcommand{\df}{\colon}
\newcommand{\id}{\operatorname{id}}

\newcommand{\cB}{{\mathcal B}}

\newcommand{\cM}{{\mathcal M}}

\newcommand{\cO}{{\mathcal O}}

\newcommand{\cU}{{\mathcal U}}

\newcommand{\g}{\mathfrak{g}}
\newcommand{\m}{\mathfrak{m}}
\newcommand{\n}{\mathfrak{n}}

\newcommand{\ba}{\mathbf{a}}
\newcommand{\bb}{\mathbf{b}}

\newcommand{\bd}{\mathbf{d}}
\newcommand{\be}{\mathbf{e}}

\newcommand{\bm}{{\mathbf m}}

\newcommand{\br}{{\mathbf r}} 
\newcommand{\bs}{{\mathbf s}}

\newcommand{\bu}{{\mathbf u}}

\newcommand{\Ome}{\Omega}    
\newcommand{\alp}{\alpha}    
\newcommand{\eps}{\epsilon}

\newcommand{\vep}{\varepsilon}
\newcommand{\la}{\lambda}

\newcommand{\rk}{\operatorname{rank}}

\newcommand{\md}{\operatorname{mod}}

\newcommand{\rep}{\operatorname{rep}}

\newcommand{\pdim}{\operatorname{proj.dim}}
\newcommand{\idim}{\operatorname{inj.dim}}

\newcommand{\rkv}{\underline{\rk}}

\newcommand{\proj}{\operatorname{proj}}

\newcommand{\Hom}{\operatorname{Hom}}
\newcommand{\Ext}{\operatorname{Ext}}
\newcommand{\ext}{\operatorname{ext}}
\newcommand{\End}{\operatorname{End}}

\newcommand{\Ima}{\operatorname{Im}}
\newcommand{\Ker}{\operatorname{Ker}}

\newcommand{\ov}{\overline}
\newcommand{\ul}{\underline}

\newcommand{\bil}[1]{\langle #1\rangle}

\newcommand{\bim}[3]{\prescript{}{#1}{#2}_{#3}}

\newcommand{\bsm}{\begin{smallmatrix}}
\newcommand{\esm}{\end{smallmatrix}}

\newcommand{\bbsm}{\left[\begin{smallmatrix}}
\newcommand{\besm}{\end{smallmatrix}\right]}

\newcommand{\bbm}{\begin{matrix}}
\newcommand{\ebm}{\end{matrix}}

\newcommand{\diag}{\operatorname{diag}}

\newcommand{\GL}{\operatorname{GL}}

\newcommand{\repvp}{\operatorname{rep_{\rm l.f.}}}
\newcommand{\indvp}{\operatorname{ind_{\rm l.f.}}}

\newcommand{\vp}{{\rm l.f.}}
\newcommand{\fib}{{\rm fib}}

\newcommand{\ra}{\rightarrow}

\newcommand{\bbo}{\mathbbm{1}}

\newcommand{\ube}{{\underline{\be}}}
\newcommand{\ubr}{{\underline{\br}}}
\newcommand{\ovH}{\overline{H}}
\newcommand{\Flf}{\operatorname{Flf}}
\newcommand{\RepFlf}{\operatorname{repFlf}}
\newcommand{\Grlf}{\operatorname{Grlf}}
\newcommand{\Grf}{\operatorname{Grf}}
\newcommand{\Gr}{\operatorname{Gr}}
\newcommand{\inc}{\operatorname{inc}}

\newcommand{\Mono}{\operatorname{Mono}}
\newcommand{\Repm}{\rep_\vp^{\mathrm{Mon}}}
\newcommand{\Hrepm}{\operatorname{HomRep}_\vp^{\mathrm{Mon}}}
\newcommand{\IHrepm}{\operatorname{IHomRep}_\vp^{\mathrm{Mon}}}

\begin{document}

\title[Quivers with relations for symmetrizable Cartan matrices II]{Quivers with relations for symmetrizable Cartan\\ matrices II: Change of symmetrizers}

\author{Christof Geiß}
\address{Christof Geiß \newline
Instituto de Matemáticas\newline
Universidad Nacional Autónoma de México\newline
Ciudad Universitaria\newline
04510 México D.F.\newline
Mexico}
\email{christof@math.unam.mx}

\author{Bernard Leclerc}
\address{Bernard Leclerc\newline
Normandie Univ, Unicaen, CNRS, LMNO, 14000 Caen France\newline
Institut Universitaire de France}
\email{bernard.leclerc@unicaen.fr}

\author{Jan Schröer}
\address{Jan Schröer\newline
Mathematisches Institut\newline
Universität Bonn\newline
Endenicher Allee 60\newline
53115 Bonn\newline
Germany}
\email{schroer@math.uni-bonn.de}

\date{December 27, 2016}

\begin{abstract}
For $k \ge 1$ we consider the $K$-algebra $H(k) := H(C,kD,\Ome)$ 
associated to a symmetrizable Cartan matrix $C$, a symmetrizer $D$
and an orientation $\Ome$ of $C$, which was defined in \cite{GLS1}. 
We construct and analyse a reduction functor from $\rep(H(k))$ to 
$\rep(H(k-1))$.
As a consequence
we show that the canonical decomposition of rank vectors
for $H(k)$ does not depend on $k$, and that
the rigid locally free $H(k)$-modules are up to isomorphism in bijection with the
rigid locally free $H(k-1)$-modules.
Finally, we show that for a rigid locally free $H(k)$-module of a given rank 
vector the Euler characteristic of the variety of flags of locally free 
submodules with fixed ranks of the subfactors does not depend on the choice of 
$k$.
\end{abstract}

\maketitle

\setcounter{tocdepth}{1}
\numberwithin{equation}{section}
\tableofcontents

\parskip2mm


\section{Introduction} \label{sec:not}


\subsection{Overview}
In \cite{GLS1} we have started to investigate the representation theory of a 
new class of quiver algebras associated with symmetrizable generalized Cartan 
matrices. One of the motivations is to obtain new geometric constructions of the
positive part of a symmetrizable Kac-Moody algebra in terms of varieties of 
representations of these quiver algebras.

An interesting feature of our construction is that for a given generalized 
Cartan matrix we obtain not just one, but an infinite series $H(k)\ (k\ge 1)$ 
of quiver algebras (the definition of $H(k)$ is recalled below). In this paper 
we study the dependence on $k$ of the representation theory of $H(k)$.
More precisely, for $k\ge 2$ we introduce a reduction functor from the module 
category $\rep(H(k))$ to the module category $\rep(H(k-1))$.
Although $\rep(H(k))$ has more indecomposable objects than $\rep(H(k-1))$, 
we show that this functor induces a bijection between isomorphism classes of 
rigid locally free $H(k)$-modules and isomorphism classes of rigid locally 
free $H(k-1)$-modules. Recall that a module $M$ over an algebra $A$ is 
\emph{rigid} if $\Ext_A^1(M,M) = 0$.
Moreover, we introduce a notion of canonical decomposition of rank vectors for 
locally free $H(k)$-modules, and we show that it does not depend on $k$. 

Let $M$ be a rigid locally free $H(k)$-module. 
Given a sequence of rank vectors $\ubr=(\br_1,\ldots,\br_l)$, we define the 
quasi-projective variety $\Flf_\ubr^{H(k)}(M)$ of flags of locally free 
submodules of $M$ with prescribed rank vectors 
$\br_1, \br_1+\br_2, \ldots, \br_1+\cdots +\br_l$. 
We show that $\Flf_\ubr^{H(k)}(M)$ is smooth and irreducible. Moreover, if 
$\overline{M}$ is the unique rigid locally free $H(k-1)$-module with the same 
rank vector as $M$, we prove that there is a canonical surjective morphism 
$\Flf_\ubr^{H(k)}(M) \to \Flf_\ubr^{H(k-1)}(\overline{M})$ 
with all fibers being affine spaces. In particular, choosing $\C$ as base field, 
we get that the topological Euler characteristics 
$\chi\left(\Flf_\ubr^{H(k)}(M)\right)$ and 
$\chi\left(\Flf_\ubr^{H(k-1)}(\overline{M})\right)$
are equal.

This last result will be used in \cite{GLS2} to realize the enveloping algebra 
of the positive part of a simple complex Lie algebra as a convolution algebra
of constructible functions over varieties of representations of $H(k)$.
In particular this convolution algebra is independent of $k$.
In another work in preparation we plan to use it also for the construction of 
new cluster characters for skew-symmetrizable acyclic cluster algebras.

The study of canonical decompositions of dimension vectors was started by 
Kac~\cite{K2} for representations of quivers over an algebraically closed field 
(i.e.~for symmetric Cartan matrices and $k=1$). Schofield~\cite{Sc} and 
Crawley-Boevey~\cite{CB} showed that this decomposition is independent of the 
base field. 
The quiver Grassmannians of rigid representations of acyclic quivers were 
studied by Caldero and Reineke \cite{CR}, who proved their smoothness.
Later Wolf~\cite{W} proved the smoothness and irreducibility of quiver flag 
varieties of rigid quiver representations.

We will now give more precise statements of our results.

\subsection{Definition of $H(C,D,\Omega)$} \label{ssec:DefH}
We use the notation of \cite{GLS1}. 
Let $C = (c_{ij})\in M_n(\Z)$ be a symmetrizable generalized Cartan matrix, and 
let $D=\diag(c_1,\ldots,c_n)$ be a symmetrizer of $C$. 
This means that $c_i \in \Z_{>0}$, and
\[
c_{ii} = 2, \qquad\qquad
 c_{ij} \le 0\quad\mbox{for}\quad i\not = j,\qquad\qquad
 c_ic_{ij} = c_jc_{ji}.
\]
When $c_{ij} < 0$ define
\[
g_{ij} := |\gcd(c_{ij},c_{ji})|,\qquad
f_{ij} := |c_{ij}|/g_{ij}.
\]
An \emph{orientation} of $C$ is a subset 
$\Omega \subset  \{ 1,2,\ldots,n \} \times \{ 1,2,\ldots,n \}$
such that the following hold:
\begin{itemize}

\item[(i)]
$\{ (i,j),(j,i) \} \cap \Omega \not= \varnothing$
if and only if $c_{ij}<0$;

\item[(ii)]
For each sequence $((i_1,i_2),(i_2,i_3),\ldots,(i_t,i_{t+1}))$ with
$t \ge 1$ and $(i_s,i_{s+1}) \in \Omega$ for all $1 \le s \le t$ we have
$i_1 \not= i_{t+1}$.

\end{itemize}
For an orientation $\Omega$ of $C$ let
$Q := Q(C,\Omega) := (Q_0,Q_1)$ be the quiver with 
vertex set $Q_0 := \{ 1,\ldots, n\}$ and 
with arrow set  
\[
Q_1 := \{ \alp_{ij}^{(g)}\df j \to i \mid (i,j) \in \Omega, 1 \le g \le g_{ij} \}
\cup  \{ \vep_i\df i \to i \mid 1 \le i \le n \}.
\]
Throughout let $K$ be an algebraically closed field.
Let
\[
H := H(C,D,\Omega) := KQ/I
\] 
where $KQ$ is the path algebra of $Q$, and $I$ is the ideal of $KQ$
defined by the following
relations:
\begin{itemize}

\item[(H1)] 
For each $i$ we have 
\[
\vep_i^{c_i} = 0;
\]

\item[(H2)]
For each $(i,j) \in \Omega$ and each $1 \le g \le g_{ij}$ we have
\[
\vep_i^{f_{ji}}\alp_{ij}^{(g)} = \alp_{ij}^{(g)}\vep_j^{f_{ij}}.
\]

\end{itemize}
This definition is illustrated by many examples in \cite[Section~13]{GLS1}.

Let $H_i$ be the subalgebra of $H$ generated by $\vep_i$.
Thus $H_i$ is isomorphic to the truncated polynomial ring $K[x]/(x^{c_i})$.
In fact, $\eps_i^{c_i}=0\in H$ by (H1). On the other hand the
only oriented cycles in $Q$ consist of the loops $\eps_i$ by the definition
of $\Ome$. Thus, since the relations (H2) involve arrows which are not loops, 
$\eps_i^{c_i-1}$ does not belong to the ideal generated by the relations.

Let $M= (M_i,M(\vep_i),M(\alpha_{ij}^{(g)}))_{i \in Q_0,(i,j) \in \Omega}$ be
a representation of $H$.
Thus $M_i$ is a finite-dimensional $K$-vector space, and 
$M(\vep_i)\df M_i \to M_i$ and
$M(\alpha_{ij}^{(g)})\df M_j \to M_i$ are $K$-linear maps satisfying the
relations (H1) and (H2).
Each $M_i$ can obviously be seen as an $H_i$-module.
We call $M$ \emph{locally free} if $M_i$ is a free $H_i$-module for 
all $i$.
In this case, $\rkv(M) = \rkv_H(M) := (\rk(M_1),\ldots,\rk(M_n))$ is the 
\emph{rank vector} of $M$, where $\rk(M_i)$ denotes the rank of the free 
$H_i$-module $M_i$.
Recall from \cite{GLS1} that an $H(k)$-module $M$ is locally free
if and only if $\pdim(M) \le 1$.

Let $\bil{-,-}_H\df \Z^n \times \Z^n \to \Z$
be the bilinear form defined by
\[
\bil{\ba,\bb}_H := \sum_{i=1}^n c_ia_ib_i + \sum_{(i,j) \in \Omega}
c_ic_{ij}a_jb_i
\]
where $\ba = (a_1,\ldots,a_n)$ and $\bb = (b_1,\ldots,b_n)$.
For locally free $H$-modules $M$ and $N$ we have
\[
\bil{\rkv(M),\rkv(N)}_H = \dim \Hom_H(M,N) - \dim \Ext_H^1(M,N),
\]
see \cite[Section~4]{GLS1}.

For $k \ge 1$ and $H = H(C,D,\Omega)$
we define $H(k) := H(C,kD,\Omega)$.
Observe that  
\[
\vep := \sum_{i=1}^n \vep_i^{c_i}
\]
belongs to the center of $H(k)$,
and $H(k)/(\vep^{k-1} H(k))\cong H(k-1)$ for $k\geq 2$.

\subsection{Canonical decompositions}
We study the affine varieties $\repvp(H(k),\br)$ of locally free $H(k)$-modules 
with rank vector $\br$.

We show that $\repvp(H(k),\br)$ is smooth and irreducible and we
compute its dimension.
A tuple $(\br_1,\ldots,\br_t)$ of rank vectors with 
$\br = \br_1 + \cdots + \br_t$ is the $H(k)$-\emph{canonical
decomposition} of $\br$ if there is a dense open subset $\cU$ of 
$\repvp(H(k),\br)$ such that each $M \in \cU$ is of the form
$
M = M_1 \oplus \cdots \oplus M_t
$
with $M_i$ an indecomposable locally free $H(k)$-module with rank vector 
$\br_i$ for $1 \le i \le t$.
Using \cite{CBS}, we show that such an $H(k)$-canonical decomposition exists and
is unique up to permutation of its entries (see Section~\ref{can-dec}).

For $k \ge 2$ we construct and study the reduction functor
\[
R\df \rep(H(k)) \to \rep(H(k-1)), \quad M \mapsto M/\vep^{k-1}M 
\]
and its restriction to locally free modules (see Proposition~\ref{prop-rigidred}).
We show that $R$ induces a natural bijection between isomorphism classes of
locally free rigid $H(k)$-modules and locally free rigid $H(k-1)$-modules.
The functor $R$ also allows us to compare $H(k)$-canonical decompositions with 
$H(k-1)$-canonical decompositions.
We obtain the following result (see Section~\ref{subsect-indep-k}):

\begin{Thm}\label{thmintro-canonical}
For $H(k) = H(C,kD,\Omega)$ the $H(k)$-canonical decompositions of
rank vectors do not depend on $k$.
\end{Thm}

\subsection{Flag varieties}
We fix for some $l\geq 2$ a sequence of rank
vectors $\ubr=(\br_1,\ldots,\br_l)\in\NN^{I\times\{1,2,\ldots,l\}}$ and set
$\bm:=\br_1+\cdots +\br_l$.
For a locally free $H(k)$-module $M$  the \emph{quiver flag variety of 
locally free submodules} $\Flf^{H(k)}_\ubr(M)$  is 
the quasi-projective variety of flags 
\[
(0 = U_0 \subset U_1 \subset U_2 \subset \cdots 
\subset U_{l-1} \subset U_l = M)
\] 
of locally free submodules 
with $\rkv_{H(k)}(U_j/U_{j-1}) = \br_j$ for $j=1,2,\ldots,l$. 

There exists up to isomorphism at most one 
rigid locally free $H(k)$-module $M_{\bm,k}$ with rank vector $\bm$.
If there is no such module, we set $M_{\bm,k}=0$.
Thus, we may define $\Flf_\ubr^{H(k)} := \Flf_\ubr^{H(k)}(M_{\bm,k})$.
We have $R(M_{\bm,k}) \cong M_{\bm,k-1}$.

\begin{Thm}\label{thm-mainthm2}
For all $k\geq 1$ and all rank vector sequences $\ubr$ we have $\Flf_\ubr^{H(k)} \neq \varnothing$   
if and only if $\Flf_\ubr^{H(1)} \neq \varnothing$. 
In this case, the following hold:
\begin{itemize}

\item[(a)]
The variety
$\Flf_\ubr^{H(k)}$ is  smooth and irreducible  of dimension 
$k\cdot d(\ubr)$ where
\[
d(\ubr):=\sum_{a<b}\bil{\br_a,\br_b}_{H(1)}.
\]

\item[(b)]
For $k \ge 2$
the canonical morphism
\[
\pi_k\df \Flf_\ubr^{H(k)} \ra \Flf_\ubr^{H(k-1)}
\]
induced by the projection $M_{\bm,k} \to R(M_{\bm,k})$
is a fiber bundle with all fibers being affine spaces of dimension 
$d(\ubr)$.

\end{itemize}
\end{Thm}

Part~(a) will be proved in Proposition~\ref{prp:sm-irr}. Though we
cannot apply directly Wolf's results~\cite{W}, our argument here follows his ideas closely.
Part~(b), which is based on Part~(a), is proved after some preparation, 
in Section~\ref{sect-concl}.

For a complex algebraic variety $X$ let $\chi(X)$ denote the topological Euler characteristic of $X$.

\begin{Cor}\label{cor-main}
Assume that $K = \C$, and let $H(k) = H(C,kD,\Omega)$ for some
$k \ge 1$.
Then for all sequences $\ubr$ the Euler characteristic 
$\chi(\Flf_\ubr^{H(k)})$ does not depend on $k$.
\end{Cor}

\begin{proof} 
Recall, that if $\pi\df V\ra W$ is a surjective morphism between 
complex algebraic varieties with $\chi(\pi^{-1}(w))=c$ for all $w\in W$, then
$\chi(W)=c\chi(V)$. This is a special case of~\cite[Proposition~4.1.31]{Di}. 
Thus, Theorem~\ref{thm-mainthm2}(b) implies that 
$\chi(\Flf_\ubr^{H(k+1)})=\chi(\Flf_\ubr^{H(k)})$ for all $k>0$ since 
$\chi(A)=1$ for each finite-dimensional complex vector space $A$.
\end{proof}

Corollary~\ref{cor-main} has the following application. 
In \cite{GLS2} we study a convolution algebra $\cM(H)$ associated with
$H$.
For $H$ a path algebra of a quiver (i.e. if $C$ is symmetric and $D$ the identity), Schofield proved that
$\cM(H)$ is isomorphic to the enveloping algebra $U(\n)$ of the positive part $\n$ of the Kac-Moody Lie algebra $\g(C)$ associated with $C$.
We prove in \cite{GLS2} that for $C$ of Dynkin type the same
result holds. 
We first show this for minimal symmetrizers $D$ and then use
Corollary~\ref{cor-main} for the generalization to arbitrary symmetrizers.

\subsection{Notation}
Unless stated otherwise, by a \emph{module} we mean a finite-dimensional 
left module. We fix a field $K$.
Except for the basic  results in Section~\ref{sec:redf}, we
always assume that $K$ is 
algebraically closed.
Let $A$ be a $K$-algebra.
Let $\md(A)$ be the category of $A$-modules, and let
$\proj(A)$ be the subcategory of projective $A$-modules.
We write $\N$ for the natural numbers, including $0$.


\section{Reduction functors} \label{sec:redf}


\subsection{Analogy with representations of modulated graphs} \label{subsect-analogy}
It was shown in \cite[Section~5]{GLS1} that $H = H(C,D,\Omega)$ gives rise to 
a generalized modulated graph, and that the category of $H$-modules is 
isomorphic to the category of representations of this generalized modulated 
graph.

As before, let $H_i$ be the subalgebra of $H$ generated by $\vep_i$.
For $(i,j) \in \Omega$ we define
\[
{_i}H_j := H_i \,\,{\rm Span}_K(\alp_{ij}^{(g)} \mid 1 \le g \le g_{ij})\,H_j. 
\]
It is shown in \cite{GLS1} that ${_i}H_j$
is an $H_i$-$H_j$-bimodule, which is free as a left $H_i$-module
and free as a right $H_j$-module.
An $H_i$-basis of ${_i}H_j$ is given by
\[
\{ \alp_{ij}^{(g)}, \alp_{ij}^{(g)}\vep_j,\ldots,
\alp_{ij}^{(g)}\vep_j^{f_{ij}-1} \mid 1 \le g \le g_{ij} \}.
\]
In particular, we have an isomorphism ${_i}H_j \cong H_i^{|c_{ij}|}$ of left 
$H_i$-modules, and we have an isomorphism
${_i}H_j \cong H_j^{|c_{ji}|}$ of right $H_j$-modules.

The tuple $(H_i, {_i}H_j)$ with $1 \le i \le n$ and $(i,j) \in \Omega$ is 
called a \emph{generalized modulation} associated with the datum
$(C,D,\Omega)$. 
A \emph{representation} $(M_i, M_{ij})$ of this generalized modulation consists 
of a finite-dimensional $H_i$-module $M_i$ for each $1 \le i \le n$, and of an 
$H_i$-linear map
\[
M_{ij}\df {_i}H_j \otimes_{H_j} M_j \to M_i
\]
for each $(i,j) \in \Omega$. 
The representations of this generalized modulation form an abelian category 
$\rep(C,D,\Omega)$ isomorphic to the category of 
$H$-modules~\cite[Proposition~5.1]{GLS1}.  
(Here we identify the category $\md(H)$ of $H$-modules with the category 
$\rep(H)$ of representations of the quiver $Q(C,\Omega)$ satisfying the 
relations (H1) and (H2).)
Given a representation $(M_i,M_{ij})$ in $\rep(C,D,\Omega)$ the corresponding 
$H$-module
$
(M_i,M(\alp_{ij}^{(g)}),M(\vep_i))
$ 
is obtained as follows:
the $K$-linear map $M(\vep_i)\df M_i \to M_i$ is given by
\[
M(\vep_i)(m) := \vep_i m.
\]
(here we use that $M_i$ is an $H_i$-module), and
for $(i,j) \in \Omega$, the $K$-linear map 
$M(\alp_{ij}^{(g)})\df M_j \to M_i$ is defined by
\[
M(\alp_{ij}^{(g)})(m) := M_{ij}(\alp_{ij}^{(g)} \otimes m).
\]
The maps $M(\alp_{ij}^{(g)})$ and $M(\vep_i)$ satisfy the defining relations 
(H1) and (H2) of $H$ because the maps $M_{ij}$ are $H_i$-linear.

\subsection{The central subalgebra $Z(k)$}
Let $H(k) = H(C,kD,\Omega)$.
Recall that  
\[
\vep := \sum_{i=1}^n \vep_i^{c_i}
\]
belongs to the center of $H(k)$.
We denote by $Z(k)$ the subalgebra which is generated by $\vep$. 
It is easy to see that $Z(k)\cong K[X]/(X^k)$. 
Obviously, $H(k)$ is a $Z(k)$-algebra, and consequently $\rep(H(k))$ is a 
$Z(k)$-linear category.
Moreover, if $C$ is connected, and $D$ is the minimal symmetrizer, $Z(k)$ is 
in fact the center of $H(k)$. Clearly, $Z(k)$ is also  contained in the 
subalgebra
\[
S(k) := \prod_{i=1}^n H_i(k)
\] 
where $H_i(k)$ is the subalgebra of $H(k)$ generated by $\vep_i$.
Thus $H_i(k) \cong K[\vep_i]/(\vep_i^{kc_i})$. 

Since $H_i(k)$ is a free $Z(k)$-module of rank $c_i$,
one easily checks that locally free $H(k)$-modules are free as $Z(k)$-modules. 
In particular, $H(k)$ is free as a $Z(k)$-module, since all projective 
$H(k)$-modules are locally free.

\subsection{The reduction functor}

We start with the following well-known lemma:
\begin{Lem} \label{lem:wn-redn}
Let $(A,\m)$ be a commutative, local ring with residue field $\kappa$.
Suppose that $f\df A^a\ra A^b$ is an $A$-linear map. Then $f$ 
is surjective, if and only if the induced, $\kappa$-linear map
$\bar{f}\df\kappa^a\ra\kappa^b$ is surjective.
Moreover, in this case $\Ker(f)$ is a direct summand of $A^a$.
\end{Lem}

We fix  $k \geq 2$. 
For typographical reasons we abbreviate $H:=H(k)$, $H_i = H_i(k)$ and 
$\ovH := H/(\vep^{k-1}H) \cong H(k-1)$ and $\ov{H}_i = H_i(k-1)$.
There is an obvious \emph{reduction functor} 
\[
R = R(k)\df \rep(H) \ra \rep(\ovH)
\]
defined by $M \mapsto M/(\vep^{k-1}M)$.
For $M \in \rep(H)$ let $\ov{M} := R(M) = M/(\vep^{k-1}M)$. 
The functor $R$ is isomorphic to the tensor functor 
$\ov{H} \otimes_H -$ and is therefore right exact.
We collect several basic properties of the restriction of the functor $R$ 
to locally free modules.

\begin{Prop} \label{prop-red}
For $k\geq 2$ the following hold:
\begin{itemize}

\item[(a)] 
If $M$ is a locally free $H$-module with $\rkv_{H}(M) = \bm$, 
then $\ov{M}$ is a locally free $\ov{H}$-module with 
$\rkv_{\ovH}(\ov{M}) = \bm$. 
In particular, this defines an exact functor
\[
R_\vp = R_\vp(k)\df \rep_{\vp}(H) \ra \rep_{\vp}(\ovH). 
\]
Moreover, $R_\vp$ 
induces a full and dense functor $\proj(H) \to \proj(\ov{H})$.

\item[(b)]
To any chain 
$X_1 \subset X_2 \subset \cdots \subset X_l$ 
of locally free  $\ovH$-modules, there exists a chain 
$Y_1 \subset Y_2 \subset \cdots \subset Y_l$ 
of locally free $H$-modules with $\ov{Y}_j\cong X_j$ for
$1 \le j \le l$. 
In particular, $R_\vp$ is dense.

\item[(c)]
For $M,N\in\rep_\vp(H)$ the
natural map 
\[
R_\vp^1(M,N)\df \Ext^1_{H}(M,N)\ra\Ext_{\ovH}^1(\ov{M},\ov{N})
\] 
is surjective.

\item[(d)]
For $M,N \in \rep_\vp(H)$ with
$\Ext_{\ovH}^1(\ov{M},\ov{N})=0$ we have 
$\Ext^1_{H}(M,N)=0$. 
In this case, $\Hom_H(M,N)$ is
a free $Z(k)$-module, and the natural map
\[
R_\vp(M,N)\df \Hom_H(M,N) \ra \Hom_{\ovH}(\ov{M},\ov{N})
\] 
is surjective.
In particular, if $M \in \rep_\vp(H)$ is indecomposable and rigid, then
$\ov{M}$ is indecomposable and rigid.

\end{itemize}
\end{Prop}

\begin{proof}
(a) 
If $M\in\rep_\vp(H)$ and $1 \le i \le n$, then by definition $e_iM$ is
a free $H_i$-module of rank $m_i$, and
$\vep^{k-1}e_iM = \vep_i^{(k-1)c_i}M$. 
Thus $e_iM/(\vep^{k-1}e_iM)$ 
is a free
$\ov{H}_i$-module, also of rank $m_i$. 
Since $R$ is right exact, and since its restriction $R_\vp$ preserves rank
vectors, we get that $R_\vp$ is exact. 
The last claim 
follows, since projective $H$-modules are locally free, and 
$\ovH = H/(\vep^{k-1}H)$. 

(b) 
We show the case $l=2$. The general case can be done similarly,  though
with heavier notation.
 
Thus, let $U = X_1$, $\bu := \rkv_{\ovH}(U)$ and $M = X_2$, $\bm := \rkv_{\ovH}(M)$. 
After choosing
for each $(i,j)\in\Ome$ left bases ${_i}L_j$ of $\bim{i}{H}{j}$ the elements
$\bar{l} = l + \vep^{k-1}{_iH_j}$ with $l\in {_i}L_j$ are identified with a left
basis $\bar{L}_{ij}$ of $\bim{i}{\ovH}{j}$. 
After choosing $\ovH_i$-bases $\bar{x}_{i,1},\ldots,\bar{x}_{i,m_i}$ of $M_i$ such that 
$\bar{x}_{i,1},\ldots,\bar{x}_{i,u_i}$ spans $U_i$ as an $\ovH_i$-module
for all $1 \le i \le n$, with respect to these bases, the structure maps 
$M_{ij}\df\bim{i}{\ovH}{j}\otimes_{\ovH_j} M_j\ra M_i$ of $M$
become matrices of the block shape
\[
M_{ij} = \left(\begin{array}{cr} U_{ij} &*\\0 & * \end{array}\right)
\in \ovH_i^{m_i\times (|c_{ij}|m_j)} \text{ and }
U_{ij}\in \ovH_i^{e_i\times (|c_{ij}|u_j)}.
\]
Since we can see an element of $\ovH_i$ or $H_i$ just as a
truncated polynomial in $\vep_i$,
we can interpret $M_{ij}$ and $U_{ij}$ also as matrices with entries in $H_i$.
Thus we can take these matrices 
to define structure maps for
locally free $H$-modules $V$ and  $N$ with the requested
properties that $V \subseteq N$, $\ov{V} =U$ and $\ov{N} = M$.

(c) and (d)  
Observe, that in case $P$ is a projective $H$-module, the
space $\Hom_{H}(P,M)$ is for each locally free $H$-module $M$ naturally 
a free $Z(k)$-module. 
By the same token $\Hom_{\ovH}(\ov{P},\ov{M})$ is a free $Z(k-1)$-module
of the same rank. It follows that the natural map 
$R_\vp(P,M)\df \Hom_{H}(P,M)\ra\Hom_{\ovH}(\ov{P},\ov{M})$ 
is surjective. 
Now, since locally free $H$-modules have projective dimension
at most 1, we can find a projective resolution
\[
0\ra P_1\xrightarrow{p_M} P_0\ra M\ra 0,
\]
which yields by (a) also a projective resolution of $\ovH$-modules
\[
0\ra\ov{P}_1\xrightarrow{\ov{p}_M} \ov{P}_1\ra \ov{M}\ra 0.
\]
We obtain a commutative diagram with exact rows:
\[
\xymatrix@-0.2pc{ 
0\ar[r]& {\Hom_{H}(M,N)}\ar[r]\ar[d]_{R_\vp(M,N)}&
{\Hom_{H}(P_0,N)}\ar[r]^{-\circ p_M}\ar[d]_{R_\vp(P_0,N)}&
{\Hom_{H}(P_1,N)}\ar[r]\ar[d]^{R_\vp(P_1,N)}
&{\Ext^1_{H}(M,N)}\ar[r]\ar[d]^{R_\vp^1(M,N)}& 0\\
0\ar[r]& {\Hom_{\ovH}(\ov{M},\ov{N})}\ar[r]&
{\Hom_{\ovH}(\ov{P}_0,\ov{N})}\ar[r]^{-\circ \ov{p}_M}&
{\Hom_{\ovH}(\ov{P}_1,\ov{N})}\ar[r]&{\Ext^1_{\ovH}(\ov{M},\ov{N})}\ar[r]& 0.
}
\]
Now,  by the above remark  $R_\vp(P_0,N)$ and $R_\vp(P_1,N)$ are surjective. 
Thus, the $K$-linear map $R_\vp^1(M,N)$ is clearly surjective.

Next, observe that in the above diagram $-\circ\bar{p}_M$ is a $Z(k-1)$-linear 
map between free $Z(k-1)$-modules, say of rank $a$ resp.~$b$.
Similarly, $-\circ{p}_M$ is a $Z(k)$-linear map between free $Z(k)$-modules,
also of rank $a$ resp.~$b$.  
Since $R_\vp(P_0,N)$ and $R_\vp(P_1,N)$ are just   the reductions
modulo $\vep^{k-1}$ and $Z(k)$ is a commutative local ring, we conclude by
Lemma~\ref{lem:wn-redn}  that $-\circ p_M$ is surjective if and only if
$-\circ\bar{p}_M$ is surjective. 
Moreover, in this case $-\circ p_M$ splits as $Z(k)$-linear map and
$-\circ \ov{p}_M$ splits as $Z(k-1)$-linear map.
This implies that $R_{l.f.}(M,N)$ is surjective. 
Finally, if $M$ is rigid, then $\overline{M}$ is also rigid by (c). 
If moreover $M$ is indecomposable, then $\End_H(M)$ is local, and $\vep^{k-1}\End_H(M)$ 
consists of nilpotent elements, so it is contained in the radical. 
Since $R_{l.f.}(M,M)$ is surjective, this shows that  
$\End_{\overline{H}}(\overline{M})$ is local, so $\overline{M}$ is indecomposable.
\end{proof}

\subsection{Example}\label{ex:a2-1}
The following example shows that the reduction functor $R_\vp$ is usually not 
full. It also shows that in general $R_\vp$ does not preserve indecomposables.
Let $C$ be a Cartan matrix of Dynkin type 
$A_2$ with symmetrizer $D=\diag(1,1)$ and orientation $\Omega=\{(1,2)\}$. 
Thus, $H(k)$ is defined by the quiver 
\[
\xymatrix{1 \ar@(ul,dl)_{\vep_1}&\ar[l]_{\alp} 2 \ar@(ur,dr)^{\vep_2}}
\]
with relations
$\vep_1\alp - \alp\vep_2 = 0$ and $\vep_i^k = 0$ for $i = 1,2$.
Consider the locally free $H(2)$-modules
\[
E_2\df \xymatrix@-0.2cm{{2}\ar[d]_<<<{\vep_2}\\{2'}}
\hspace{.75cm} \text{ and }\hspace{.75cm}
M\df
\xymatrix@-0.2cm{1\ar[d]_<<<{\vep_1}\\1'&\ar[l]_<<<{\alp}2\ar[d]^<<<{\vep_2}\\& 2'}
\]
In the left picture, $2$ and $2'$ denote two basis vectors of $E_2$, regarded as a $K$-vector
space, and the arrow means that $M(\varepsilon_2)(2) = 2'$.
Similarly, in the right picture, $1$ and $1'$ (resp. $2$ and $2'$) denote two basis vectors in the 
space $M_1$ (resp. in the space $M_2$), and the arrows represent the linear
maps $M(\varepsilon_i)$ and $M(\alpha)$.
Thus, $\rkv_{H(2)}(E_2) = (0,1)$ and $\rkv_{H(2)}(M) = (1,1)$. 
We get $\ov{M} \cong S_1 \oplus S_2$ and $\ov{E}_2 = S_2$.
Here $S_1$ and $S_2$ denote the simple $H(1)$-modules.

Both homomorphism spaces 
$\Hom_{H(2)}(E_2,M)$ and $\Hom_{H(1)}(\ov{E}_2,\ov{M})$ are
$1$-dimensional.
The space 
$\Hom_{H(2)}(E_2,M)$ has a basis vector
$f\df E_2 \to M$ which maps the basis vector
$2$ of $E_2$ to the basis vector $2'$ of $M$.
However, we have $\vep f=0$ which implies $\ov{f}=0$.


\section{Canonical decompositions of rank vectors}


\subsection{Varieties of locally free $H$-modules}\label{subsec-varieties}
Let $\bd = (d_1,\ldots,d_n) \in \N^n$ be a dimension vector.
Let $\rep(H,\bd)$ denote the affine variety of representations of $H$ with dimension
vector $\bd$. This is acted upon by the group
\[
 G_\bd := \prod_{i=1}^n \GL_{d_i}(K).
\]
The $G_\bd$-orbits in $\rep(H,\bd)$ are naturally in bijection
with the isomorphism classes of $H$-modules with dimension vector $\bd$. 
If $M \in \rep(H,\bd)$ is locally free, its rank vector is 
$\br =(r_1,\ldots,r_n)$ where $r_i:=d_i/c_i$. Hence locally free modules can 
only exist if $d_i$ is divisible by $c_i$ for every $i$.
In this case, we say that $\bd$ is $D$-\emph{divisible}.
Let $\repvp(H,\br)$ be the union of all $G_\bd$-orbits $\cO_M$ of locally free
modules $M$ of rank vector $\br$.
Consider the natural projection 
\[
\pi\ \df\  \repvp(H,\br) \to \rep(H_1,d_1) \times \cdots \times \rep(H_n,d_n)
\]
defined by
\[
(M(\alp))_{\alp\in Q_1} \mapsto (M(\vep_1),\ldots, M(\vep_n)).
\]
The image of $\pi$ is 
$\cO_{H_1^{r_1}} \times \cdots \times \cO_{H_n^{r_n}}$,
where $\cO_{H_i^{r_i}}$ is the $G_{d_i}$-orbit of the free $H_i$-module 
$H_i^{r_i}$ of rank $r_i$.
(Note that $\rep(H_i,d_i)$ is just a point if $c_i=1$.)

The free $H_i$-module $H_i$ can also be seen as an $H$-module
which we denote by $E_i$. 
We identify $\Ima(\pi)$ with the $G_\bd$-orbit $\cO_{E^{\br}}$
of the locally free $H$-module 
\[
E^\br := \bigoplus_{i=1}^n E_i^{r_i}.
\] 
In particular,
$\cO_{E^{\br}}$ is smooth and irreducible of dimension
\[
\sum_{i=1}^n c_i^2r_i^2 - \sum_{i=1}^n c_ir_i^2.
\] 
Here, the summands of the first sum are the dimensions of the groups $G_{d_i}$,
while the summands of the second sum are the dimensions of the endomorphism
rings $\End_H(E^{\br}) \cong \prod_{i=1}^n\End_{H_i}(H_i^{r_i})$. 

By Section~\ref{subsect-analogy}, the fibre 
$\pi^{-1}(E_1^{r_1},\ldots,E_n^{r_n})$ can be identified with
\[
\rep_\vp^\fib(H,\br) := 
\prod_{(i,j) \in \Omega}\ 
\Hom_{H_i}\left({_i}H_j \otimes_{H_j} E_j^{r_j},\ E_i^{r_i}\right)
\cong
\prod_{(i,j) \in \Omega}\ 
\Hom_{H_i}\left( E_i^{|c_{ij}|r_j},\ E_i^{r_i}\right).
\]

The group 
\[
G(H,\br) := \prod_{i=1}^n \GL_{r_i}(H_i)
\] 
acts on $\rep_\vp^\fib(H,\br)$ via
\[
(g \cdot M)_{ij} = g_i M_{ij}({_iH_j} \otimes g_j^{-1}).
\] 
The $G(H,\br)$-orbits in $\rep_\vp^\fib(H,\br)$ are naturally in bijection
with the isomorphism classes of locally free $H$-modules with rank vector $\br$.
We have
\[
\dim\rep_\vp^\fib(H,\br) = \dim_K G(H,\br) - q_H(\br) = 
\sum_{(i,j) \in \Omega} c_i|c_{ij}|r_ir_j.
\]
Here and for later use we abbreviate $q_H(\br):=\bil{\br,\br}_H$, 
see Section~\ref{ssec:DefH}.

\begin{Prop}\label{fibredim1}
Let $\bd = (d_1,\ldots,d_n)$ be $D$-divisible as above. 
Set $r_i := d_i/c_i$ and  $\br = (r_1,\ldots,r_n)$. 
Then  $\repvp(H,\br)$ is a non-empty open subset of $\rep(H,\bd)$. Moreover
we have:
\begin{itemize}

\item[(i)]
The restriction ${\bar{\pi}}\,\df \repvp(H,\br)\ra \cO_{E^{\br}}$ of $\pi$ to
its image defines a vector bundle of rank 
$\sum_{(i,j) \in \Omega} c_i|c_{ij}|r_ir_j$.  
In particular, $\repvp(H,\br)$ is smooth and irreducible of dimension
\[
\sum_{i=1}^n c_i(c_i-1) r_i^2 + \sum_{(i,j) \in \Omega} c_i|c_{ij}|r_ir_j =
\dim(G_\bd) - q_H(\br).
\]

\item[(ii)]
If $q_H(\br)\le 0$ then $\repvp(H,\br)$ has infinitely many $G_\bd$-orbits.

\end{itemize}
\end{Prop}

\begin{proof}
 The function 
\[
l_\br\df\rep(H,\bd)\ra\NN,\;\;\; M\mapsto \sum_{i=1}^n\rk(M(\varepsilon_i))
\]
is lower semicontinuous. Now, $M\in\rep(H,\bd)$ is locally free if and
only if $l_\br(M)$ takes the maximum $\sum_{i=1}^n (d_i- d_i/c_i)$. This shows,
that the locally free modules form an open subset of $\rep(H,\bd)$.

Next, notice that $\bar{\pi}$ is by construction $G_{\bd}$-equivariant. Since
$\cO_{E^{\br}}$ is a single $G_\bd$-orbit, all fibers of $\bar{\pi}$ are 
isomorphic, and, in particular, are vector spaces of the same
dimension.

Consider the trivial vector bundle  
\[
X:= \left(\prod_{\alp\in Q_1^\circ}\Hom_K(K^{d_{s(\alp)}},K^{d_{t(\alp)}})\right)
\times \cO_{E^{\br}}.
\]
over $\cO_{E^{\br}}$. 
A point of $X$ is given by a tuple 
$
M = \left(\left(M(\alp_{ij}^{(g)})\right)_{(i,j)\in\Omega; 1\leq g\leq g_{ij}},\, 
(M(\vep_i))_{1 \le i \le n}\right)
$
of $K$-linear maps.
Obviously, the map $\mu\df X \ra X$ defined by
\[
\mu(M) :=
\left(\left(M(\vep_i)^{f_{ji}} M(\alp_{ij}^{(g)})-M(\alp_{ij}^{(g)})M(\vep_j)^{f_{ij}}\right),\, (M(\vep_i))\right)
\]
is an endomorphism of the vector bundle $X$, and by construction
$\Ker(\mu) = \repvp(H,\br)$. 
Since by the above consideration, the fibre
\[
\Ker(\mu)_{(M(\vep_i))} = \bar{\pi}^{-1}(M(\vep_i))
\]
is of constant dimension for all $(M(\vep_i)) \in\cO_{E^{\br}}$,
we have that $\Ker(\mu)$ is a vector bundle  of the claimed rank over $\cO_{E^{\br}}$.
This proves (i).

The $1$-dimensional torus 
$\{(\la\id_{d_1},\ldots,\la\id_{d_r})\mid \la \in K^*\}\subset G_\bd$ acts
trivially on the variety $\repvp(H,\br)$. 
So the maximal dimension of a $G_\bd$-orbit is $\dim(G_\bd) - 1$. 
Hence, by (i), if $q_H(\br)\le 0$, every $G_\bd$-orbit has dimension at most 
$\dim(\repvp(H,\br))-1$. 
This proves (ii).
\end{proof}

The vector bundle structure of 
Proposition~\ref{fibredim1} is inspired by \cite[Section~2]{B}.

\subsection{Rigid modules}

\begin{Prop}\label{prop-rigidred}
For $k\geq 2$ the reduction functor 
\[
R_\vp\df\rep_\vp(H)\ra\rep_\vp(\ovH)
\]
induces a bijection between the 
isomorphism classes of rigid locally free $H$-modules and rigid locally free $\ovH$-modules.
\end{Prop}

\begin{proof}
By Proposition~\ref{fibredim1} we know that $\repvp(H,\br)$ is irreducible
for all rank vectors $\br$.
If $M$ is a rigid locally free $H$-module with $\rkv(M) = \br$, then its orbit in $\repvp(H,\br)$ is open and dense.
In particular, for a given rank vector $\br$ there is at most one rigid
module with rank vector $\br$ up to isomorphism.
Now the result follows from Proposition~\ref{prop-red}(d). 
\end{proof}

\subsection{Canonical decompositions} \label{can-dec}
Let $\br = (r_1,\ldots,r_n)$ be a rank vector and let $\bd = (c_1r_1,\ldots,c_nr_n)$ be the corresponding dimension vector.
Let $\indvp(H,\br)$ be the constructible subset of $\rep_\vp(H,\br)$
consisting of all indecomposable locally free $H$-modules with rank
vector $\br$.
The rank vector $\br$ is an $H$-\emph{Schur root} if 
$\indvp(H,\br)$ is dense in $\repvp(H,\br)$.

For any tuple $(\br_1,\ldots,\br_t)$ of rank vectors such that
$\br_1 + \cdots + \br_t = \br$ there is a morphism of quasi-projective varieties
\[
G_\bd  \times \repvp(H,\br_1) \times \cdots \times 
\repvp(H,\br_t) \to \repvp(H,\br)
\]
defined by
$(g,M_1,\ldots,M_t) \mapsto g \cdot (M_1 \oplus \cdots \oplus M_t)$.
Let
\[
\eta_{(\br_1,\ldots,\br_t)}\df G_\bd  \times \indvp(H,\br_1) \times \cdots \times 
\indvp(H,\br_t) \to \repvp(H,\br)
\]
be its restriction to $G_\bd  \times \indvp(H,\br_1) \times \cdots \times 
\indvp(H,\br_t)$.

By Proposition~\ref{fibredim1} we know that the variety $\repvp(H,\br)$ is irreducible.
Thus, up to permutation of its entries, there is a unique tuple
$(\br_1,\ldots,\br_t)$ of rank vectors such that the image of
$\eta_{(\br_1,\ldots,\br_t)}$ is dense in $\repvp(H,\br)$.
In this case, it follows that 
the rank vectors $\br_1,\ldots,\br_t$ are $H$-Schur roots, compare
the proof of \cite[Theorem~1.1]{CBS}.
We call $(\br_1,\ldots,\br_t)$ the $H$-\emph{canonical decomposition} of $\br$.

In contrast to the special case with $C$ symmetric and $D$ the identity
matrix, one can in general not expect that for an $H$-Schur root $\br$
there exists a module $M \in \indvp(H,\br)$ with $\End_H(M) = K$.

For rank vector $\br$ and $\bs$ for $H$ set
\[
\ext_H(\br,\bs) := \min\{ \dim \Ext_H^1(M,N) \mid (M,N) \in \repvp(H,\br) \times \repvp(H,\bs) \}.
\]
The function $\dim \Ext_H^1(-,?)$ is upper semicontinuous.
Therefore the set of all $(M,N) \in \repvp(H,\br) \times \repvp(H,\bs)$
with $\dim \Ext_H^1(M,N) = \ext_H(\br,\bs)$ is open in
$\repvp(H,\br) \times \repvp(H,\bs)$.

Let $\bd$ and $\bd_1,\ldots,\bd_t$ be dimension vector for $H$ such that 
$\bd = \bd_1 + \cdots + \bd_t$.
For constructible subsets $C_i \subseteq \rep(H,\bd_i)$ with $1 \le i \le t$
let 
\[
C_1 \oplus \cdots \oplus C_t := \{ M \in \rep(H,\bd) \mid 
M \cong U_1 \oplus \cdots \oplus U_t \text{ with }
U_i \in C_i \text{ for all } 1 \le i \le t 
\}.
\]
Let $\ov{C_1 \oplus \cdots \oplus C_t}$ be the Zariski closure of 
$C_1 \oplus \cdots \oplus C_t$ in $\rep(H,\bd)$.
The following result generalizes \cite[Proposition~3(a)]{K2}.

\begin{Thm}\label{thm-canonical1}
For $H = H(C,D,\Omega)$, a tuple
$(\br_1,\ldots,\br_t)$ of rank vector is the $H$-canonical decomposition of the rank vector $\br := \br_1 + \cdots + \br_t$ if
and only if the following hold:
\begin{itemize}

\item[(i)]
$\br_i$ is an $H$-Schur root for all $1 \le i \le t$.

\item[(ii)]
$\ext_H(\br_i,\br_j) = 0$ for all $i \not= j$.

\end{itemize}
In this case, we have
\[
\ov{\indvp(H,\br_1) \oplus \cdots \oplus \indvp(H,\br_t)}
=
\ov{\rep_\vp(H,\br_1) \oplus \cdots \oplus \rep_\vp(H,\br_t)}
=
\ov{\rep_\vp(H,\br)}.
\]
\end{Thm}

\begin{proof}
We know that $\repvp(H,\br)$ is an open and irreducible subset of
$\rep(H,\bd)$, where $\br = (r_1,\ldots,r_n)$ and $\bd = (c_1d_1,\ldots,c_nr_n)$.
Now the theorem follows directly from \cite[Theorem~1.2]{CBS}.
\end{proof}

\subsection{Independence of $k$}\label{subsect-indep-k}

\begin{Lem}\label{lem-canonical2}
Let $H = H(k)$ and $\ov{H} = H(k-1)$.
For rank vectors $\br$ and $\bs$ we have
$\ext_H(\br,\bs) = 0$ if and only if $\ext_{\ov{H}}(\br,\bs) = 0$.
\end{Lem}

\begin{proof}
This follows from the definitions and from 
Proposition~\ref{prop-red}(c) and (d).
\end{proof}

\begin{Lem}\label{lem-canonical3}
Let $H = H(k)$ and $\ov{H} = H(k-1)$.
A rank vector $\br$ is an $H$-Schur root if and only if $\br$ is
an $\ov{H}$-Schur root.
\end{Lem}

\begin{proof}
Assume that $\indvp(H,\br)$ is dense in $\repvp(H,\br)$.
Let $(\br_1,\ldots,\br_t)$ be the $\ov{H}$-canonical decomposition of
$\br$.
To get a contradiction, we assume that $t \ge 2$.
By induction we get that the $\br_i$ are $H$-Schur roots.
From Theorem~\ref{thm-canonical1} we know that $\ext_{\ov{H}}(\br_i,\br_j) = 0$ for all
$i \not= j$.
Now Lemma~\ref{lem-canonical2} implies that $\ext_H(\br_i,\br_j) = 0$ for all $i \not= j$.
Again from Theorem~\ref{thm-canonical1} we get that
\[
\ov{\repvp(H,\br_1) \oplus \cdots \oplus \repvp(H,\br_t)} =
\ov{\repvp(H,\br)}.
\]
This is a contradiction, since we assumed that $\br$ is an $H$-Schur root.

To show the other direction, let us now assume that $\br$ is
an $\ov{H}$-Schur root.
Let $(\br_1,\ldots,\br_t)$ be the $H$-canonical decomposition of
$\br$.
To get a contradiction, we assume that $t \ge 2$.
By induction we get that the $\br_i$ are $\ov{H}$-Schur roots.
From Theorem~\ref{thm-canonical1} we know that $\ext_H(\br_i,\br_j) = 0$ for all
$i \not= j$.
Thus Lemma~\ref{lem-canonical2} 
implies that $\ext_{\ov{H}}(\br_i,\br_j) = 0$ for all $i \not= j$.
Again from Theorem~\ref{thm-canonical1} we get that
\[
\ov{\repvp(\ov{H},\br_1) \oplus \cdots \oplus \repvp(\ov{H},\br_t)} =
\ov{\repvp(\ov{H},\br)}.
\]
This is a contradiction, since we assumed that $\br$ is an $\ov{H}$-Schur root.
\end{proof}

As a direct consequence of Lemmas~\ref{lem-canonical2} and
\ref{lem-canonical3} we get the following result.

\begin{Thm}\label{thm-canonical4}
Let $H = H(C,D,\Omega)$, and let $\br$ be a rank vector for $H$.
Then the $H$-canonical decomposition of $\br$ does not depend on the
symmetrizer $D$.
\end{Thm}

\subsection{The symmetric case}
Let $C \in M_n(\Z)$ be a symmetric Cartan matrix, and let $D$ be its minimal symmetrizer.
Thus $D$ is just the identity matrix.
Furthermore, let $\Omega$ be an orientation of $C$.
As in \cite{GLS1} let $Q = Q(C,\Omega)$ and $Q^\circ = Q \setminus 
\{ \vep_1,\ldots,\vep_n \}$.
Set $H(k) = H(C,kD,\Omega)$ for some $k \ge 1$.
It follows that
\[
H(k) \cong KQ^\circ \otimes_K K[X]/(X^k).
\]
In particular, we have $H(1) \cong KQ^\circ$.
Now for $K$ an algebraically closed field one can
combine Schofield's algorithm \cite{Sc} with Theorem~\ref{thm-canonical4} to 
compute the $H(k)$-canonical decomposition of each rank vector.
(Schofield \cite{Sc} is using the extra assumption that the characteristic of 
$K$ is zero.
It follows from \cite{CB} that this assumption can be dropped.)


\section{Varieties of flags of locally free modules} \label{sect4}


\subsection{Quiver flag varieties}
Recall that
for a locally free $H(k)$-module $M$  the \emph{quiver flag variety of 
locally free submodules} $\Flf^{H(k)}_\ubr(M)$  is 
the quasi-projective variety of flags 
\[
(0 = U_0 \subset U_1 \subset U_2 \subset \cdots 
\subset U_{l-1} \subset U_l = M)
\] 
of locally free submodules 
with $\rkv_{H(k)}(U_j/U_{j-1}) = \br_j$ for $j=1,2,\ldots,l$. 
Note that in case $U$ is a submodule of the locally free module $M$,
then $U$ is locally free if and only if $M/U$ is locally free 
by~\cite[Lemma~3.8]{GLS1}. In particular, we could have defined 
$\Flf^{H(k)}_{\ubr}(M)$ alternatively via the rank vectors of the 
submodules $U_i$.

In the special case $l=2$ we write 
\[
\Grlf_\be^{H(k)}(M) := \Flf_{(\be,\bm-\be)}^{H(k)}(M).
\] 
This is called the 
\emph{quiver Grassmannian of locally free submodules} of rank vector
$\be$ of $M$. 

The  aim of this section is to prove Theorem~\ref{thm-mainthm2}.
For Part (a), we closely follow the ideas of
Wolf's Thesis \cite[Section~5]{W}, see Proposition~\ref{prp:sm-irr}.

\subsection{The algebra $H(k,l)$}
For $k \ge 1$ and $l \ge 2$ 
let $A_l$ be the path algebra of the linearly oriented quiver
$1\ra 2\ra\cdots\ra (l-1)$, and let 
\[
H(k,l) := H(k)\otimes_K A_l.
\]
The description of a tensor product of paths algebras
of quivers with relations can be found in \cite[Section~1]{L}.

Thus, we can think of an $H(k,l)$-module $\ul{M}$ as a tuple
\begin{multline*}
(M_1,\ldots,M_{l-1};\mu_1,\ldots,\mu_{l-2})\\ 
\in\rep(H(k))^{l-1}\times
\Hom_{H(k)}(M_1,M_2)\times\cdots\times\Hom_{H(k)}(M_{l-2},M_{l-1}).
\end{multline*}
In this language,
\begin{multline*}
\Hom_{H(k,l)}(\ul{M},\ul{M}')=\{(f_1,\ldots,f_{l-1})\in\\
\Hom_{H(k)}(M_1,M'_1)\times\cdots\times\Hom_{H(k)}(M_{l-1}, M'_{l-1})\mid\\
\mu'_i f_i=f_{i+1}\mu_i\text{ for all } 1\leq i\leq l-2\}.
\end{multline*}
We call $\ul{M}$ \emph{locally free}, if all $M_1,\ldots, M_{l-1}$ are
locally free $H(k)$-modules. 
In this case we write
\[
\rkv_{H(k,l)}(\ul{M}) = (\rkv_{H(k)}(M_1),\ldots,\rkv_{H(k)}(M_{l-1})).
\]
The following proposition is inspired 
from~\cite[Proof of Theorem~5.27, Appendix~B3]{W}.

\begin{Prop}\label{prp:bilhkl}
Suppose 
\[
\ul{M} = (M_1,\ldots,M_{l-1};\mu_1,\ldots,\mu_{l-2})
\text{\;\;\; and \;\;\;}
\ul{M}' = (M'_1,\ldots,M'_{l-1};\mu'_1,\ldots,\mu'_{l-2})
\] 
belong to $\repvp(H(k,l))$.
\begin{itemize}

\item[(a)]
We have
$\pdim(\ul{M}) \leq 2$ and $\idim(\ul{M}) \le 2$.

\item[(b)]
If all $\mu_1,\ldots,\mu_{l-2}$ are  injective, then $\pdim(\ul{M}) \leq 1$.

\item[(c)]
If  all $\mu_1,\ldots,\mu_{l-2}$ are 
surjective, then $\idim(\ul{M}) \leq 1$.

\item[(d)] 
The value of the homological bilinear form depends
for locally free $H(k,l)$-modules only on the rank vector. 
More precisely, we have
\begin{align*}
\bil{\ul{M},\ul{M}'}_{H(k,l)} &:=
\sum_{i=0}^2 (-1)^i\dim_K\Ext^i_{H(k,l)}(\ul{M},\ul{M}')\\
&= \sum_{j=1}^{l-1}\bil{M_j,M'_j}_{H(k)}-
\sum_{j=1}^{l-2}\bil{M_j,M'_{j+1}}_{H(k)}.
\end{align*}

\end{itemize}
\end{Prop}

\begin{proof}
The indecomposable projective $H(k,l)$-modules are the tensor products of 
the indecomposable projective $H(k)$-modules and the indecomposable projective 
$A_l$-modules. 
More precisely, these are the modules of the form
\[
 P_{(i,j)} = (0,\ldots,0,P_i,\ldots,P_i; 0,\ldots,0, \id_{P_i}, \ldots, \id_{P_i}),\qquad
 (1\le i \le n,\ 1\le j \le l-1),
\]
where we have $l-j$ copies of the indecomposable projective $H(k)$-module $P_i$.
Since every locally free $H(k,l)$-module has a filtration with successive 
subquotients of the form
\[
 E_{(i,j)} = (0,\ldots,0,E_i,0,\ldots,0; 0,\ldots,0),\qquad
 (1\le i \le n,\ 1\le j \le l-1),
\]
(where $E_i$ is in position $j$), it is enough to show that 
$\pdim(E_{i,j}) \le 2$.
Clearly $\pdim(E_{i,l-1}) \le 1$. For $1\le j \le l-2$ we have a projective 
resolution of the form
\[
 0 \to \bigoplus_{h\in\Omega(-,i)} P_{h,j+1}^{|c_{hi}|} 
\to P_{i,j+1}\oplus\bigoplus_{h\in\Omega(-,i)} P_{h,j}^{|c_{hi}|} \to P_{i,j} \to E_{i,j} \to 0,
 \]
compare \cite[Proposition~3.1]{GLS1}. 
The proof is similar for the injective dimension. This proves (a).

For an $H(k)$-module $M$ and $2\le j\le l$, define
\[
 \underline{M}^{(j)} := (0,\ldots,0,M,\ldots,M; 0,\ldots,0, \id_{M}, \ldots, \id_{M}), 
\]
where we have $j-1$ copies of $M$ and $j-2$ identity maps. 
The $H(k,l)$-modules described in (b) are precisely the modules which admit a 
filtration whose 
successive subquotients are isomorphic to modules of the form 
$\underline{M}^{(j)}$. 
If $0 \to P' \to P \to M \to 0$ is a projective
resolution of $M$, then 
$0 \to \underline{P'}^{(j)} \to \underline{P}^{(j)} \to \underline{M}^{(j)} \to 0$ 
is a projective
resolution of $\underline{M}^{(j)}$, hence $\pdim(\underline{M}^{(j)}) \le 1$. 
This proves (b). 
The proof of (c) is dual.

By (a) the form $\langle -, -\rangle_{H(k,l)}$ descends to the Grothendieck group
of the category of locally free $H(k,l)$-modules. 
So it is enough to verify the formula for modules of the form $E_{i,j}$.
We leave this as an exercise.
\end{proof}

For $M\in\rep_\vp(H(k))$ let
\[
\ul{M}^{(l)} := (M,\ldots,M;\id_M,\ldots,\id_M)\in\rep_\vp(H(k,l))
\]
denote the \emph{repetitive module} associated to $M$.

\begin{Lem} \label{lem:hklbil2}
Let $M \in \rep_\vp(H(k))$ be rigid. Then  for any
locally free submodule $\ul{U}$ of $\ul{M}^{(l)}$ and any
locally free factor module $\ul{F}$ of $\ul{M}^{(l)}$ we have
\[
\Ext^i_{H(k,l)}(\ul{U},\ul{F})=0 \text{ for } i=1,2. 
\]
In particular, 
\[
\dim\Hom_{H(k,l)}(\ul{U},\ul{M}^{(l)}/\ul{U})=
\sum_{a<b}\bil{\br_a,\br_b}_{H(k)}
\]
with  $\br_i:=\rkv_{H(k)}(U_i/U_{i-1})$ for $i=1,2,\ldots,l$, where we
set $U_0=0$ and $U_l=M$.
\end{Lem}

\begin{proof} 
Clearly we have $\End_{H(k,l)}(\ul{M}^{(l)}) \cong \End_{H(k)}(M)$.
By Proposition~\ref{prp:bilhkl}(b) we have
$\Ext^2_{H(k,l)}(\ul{M}^{(l)},\ul{M}^{(l)})=0$.
By Proposition~\ref{prp:bilhkl}(d), 
$\bil{\ul{M}^{(l)}, \ul{M}^{(l)}}_{H(k,l)}=\bil{M,M}_{H(k)}$. Thus, 
all together, again by Proposition~\ref{prp:bilhkl}(d),
\[
\Ext^1_{H(k,l)}(\ul{M}^{(l)},\ul{M}^{(l)})\cong\Ext^1_{H(k)}(M,M)=0.
\]
Similarly, by~Proposition~\ref{prp:bilhkl}, we have
$\Ext^2_{H(k,l)}(\ul{U},-)=0=\Ext^2_{H(k,l)}(-,\ul{F})$. Thus, we have
surjections
\[
0=\Ext^1_{H(k,l)}(\ul{M}^{(l)},\ul{M}^{(l)})\ra
\Ext^1_{H(k,l)}(\ul{U},\ul{M}^{(l)})\ra \Ext^1_{H(k,l)}(\ul{U},\ul{F}),
\]
which shows our first claim.

It follows, that we have in particular
\[
\dim\Hom_{H(k,l)}(\ul{U},\ul{M}^{(l)}/\ul{U})=
\bil{\ul{U},\ul{M}^{(l)}/\ul{U}}_{H(k,l)}.
\]
Now, the second claim is a straightforward computation with the 
formula from Proposition~\ref{prp:bilhkl}(d).
\end{proof}

\subsection{Tangent space}

\begin{Prop} \label{prp:fl-gr}
Let $\ubr=(\br_1,\ldots,\br_l)$ be a sequence of rank vectors for 
$H(k)$, and 
$\ube=\ube(\ubr):=(\br_1,\br_1+\br_2,\ldots,\br_1+\br_2+\cdots+\br_{l-1})$ the
corresponding rank vector for $H(k,l)$. Moreover, let 
$M\in\rep_\vp(H(k))$ with $\rkv_{H(k)}(M) = \bm = \sum_{i=1}^l\br_i$.
With this notation we have a natural isomorphism of schemes
\[
\iota\df\Flf_\ubr^{H(k)}(M)\ra\Grlf_\ube^{H(k,l)}(\ul{M}^{(l)})
\]
defined by
\[
U_\bullet=(U_0,U_1,\ldots,U_l)\mapsto 
(U_1,\ldots,U_{l-1};\inc_{1,2},\ldots,\inc_{l-2,l-1}),
\]
where $\inc_{i,i+1}\df U_i\hookrightarrow U_{i+1}$ is the natural 
inclusion.
Furthermore, the tangent space at $U_\bullet$ can be described as
$T_{U_\bullet}\Flf_\ubr^{H(k)}(M)=
\Hom_{H(k,l)}(\iota(U_\bullet), \ul{M}^{(l)}/\iota(U_\bullet))$.
\end{Prop}

\begin{proof} 
Recall that $D$ is the symmetrizer of the Cartan matrix $C$. For a
sequence of dimension vectors $\ubr=(\br_1,\ldots,\br_l)$, 
we will use the shorthand notation $D\ubr:=(D\br_1,\ldots, D\br_l)$.
As observed in the proof of~\cite[Lemma~5.23]{W} the scheme
$\operatorname{Fl}^{H(k)}_{D\ubr}(M)$ of flags of \emph{all} submodules of $M$
of type $D\ubr$, and the scheme 
$\operatorname{Gr}^{H(k,l)}_{D\ube}(\underline{M}^{(l)})$
of \emph{all} submodules of $\underline{M}^{(l)}$ of dimension vector $D\ube$ 
are isomorphic.  Note that Wolf's easy argument for an arbitrary base
field $K$ is valid for quivers $Q$ possibly with loops and any nilpotent
representation $M$ of $Q$.

Now, in both schemes the (flags of) locally free submodules 
form open subschemes whose $K$-rational points are in bijection under $\iota$. 
In fact, let $U\in\Gr^{H(k,l)}_{D\ube}(\ul{M}^{(l)})$ be locally free.
Then we can choose for each $U_{i,j}$ a complement $V_{i,j}$ such that 
$U_{i,j}\oplus V_{i,j}=\ul{M}^{(l)}_{i,j}$ as an $H_i$-module. Now, 
\[
\cU:=\{U'\in \Gr^{H(k,l)}_{D\ube}(\ul{M}^{(l)})\mid U'_{i,j}\cap V_{i,j}=0
\text{ for all } (i,j)\}
\]
is an open neighbourhood of $U$ which consists of locally free modules.
This proves the first claim.

For the second claim we  use  that 
$\Grlf_{\ube}^{H(k,l)}(\ul{M}^{(l)})$ is an open subset of the usual
quiver Grassmannian $\Gr^{H(k,l)}_{D\ube}(\ul{M}^{(l)})$. 
Thus, we can apply Schofield's formula for the tangent space of a
quiver Grassmannian~\cite[Lemma~3.2]{Sc}.
\end{proof}

\begin{Cor} \label{Cor:FRigTang}
Let $M$ be a rigid locally free $H(k)$-module and 
$\ubr=(\br_1,\ldots,\br_l)$ a sequence of rank vectors with 
$\sum_{i=1}^l\br_i=\rkv_{H(k)}(M)$. Then for each 
$U_\bullet\in\Flf_\ubr^{H(k)}(M)$ we have 
\[
\dim T_{U_\bullet}\Flf_\ubr^{H(k)}(M)=
\sum_{a<b}\bil{\br_a,\br_b}_{H(k)}.
\]
\end{Cor}

\begin{proof}
By Proposition~\ref{prp:fl-gr} we have 
\[
T_{U_\bullet}\Flf_\ubr^{H(k)}(M) =
\Hom_{H(k,l)}(\iota(U_\bullet), \ul{M}^{(l)}/\iota(U_\bullet)). 
\]
Since
$M$ is rigid, the dimension of the latter space is
$\sum_{a<b}\bil{\br_a,\br_b}_{H(k)}$ by Lemma~\ref{lem:hklbil2}.
\end{proof}

\subsection{Dimension and irreducibility}

\begin{Lem} \label{lem:flags}
Let $\ul{r}=(r_1,\ldots,r_l)$ be a sequence of ranks, and set 
$m:=r_1+\cdots+r_l$. 
Then
$\Flf_{\ul{r}}^{K[x]/(x^k)}((K[x]/(x^k))^m)$ is a smooth irreducible
quasi-projective variety of dimension $k\cdot d(\ul{r})$ for
$d(\ul{r}):=(\sum_{a<b} r_ar_b)$. More precisely, we have
a fiber bundle 
\[
\pi\df\Flf_{\ul{r}}^{K[x]/(x^k)}((K[x]/(x^k))^m)\ra
\operatorname{Fl}_{\ul{r}}(K^m)
\]
defined by $U \mapsto U/xU$
with fibers isomorphic to
$K^{(k-1)d(\ul{r})}$.
\end{Lem}

We leave the proof as an exercise. Note however, that $\pi$ is in
general not a vector bundle.

Recall that $S(k) = \prod_{i=1}^n H_i(k)$.
Let $\ubr=(\br_1,\ldots,\br_l)$ be a sequence of rank vectors for
$S(k)$, and set $\bm:=\br_1+\cdots+\br_l$. 
Define
\[
S(k)^\bm := \bigoplus_{i=1}^n H_i(k)^{m_i}, 
\]
and let $\Flf_{\ubr}^{S(k)} := \Flf_{\ubr}^{S(k)}(S(k)^\bm)$.

\begin{Cor} \label{cor:Sflag}
Let $\ubr=(\br_1,\ldots,\br_l)$ be a sequence of rank vectors for
$S(k)$.
Then $\Flf_{\ubr}^{S(k)}$ is a smooth irreducible quasi-projective variety of 
dimension 
\[
\sum_{a<b}\bil{\br_a,\br_b}_{S(k)}=\sum_{i=1}^n \sum_{a<b} kc_ir_{i,a}r_{i,b},
\]
where $\bil{-,-}_{S(k)}$ is the symmetric bilinear form defined by the
matrix $\diag(kc_1,\ldots,kc_n)$.
\end{Cor}

\begin{proof}
The quasi-projective variety $\Flf_{\ubr}^{S(k)}$ is a product of flag varieties 
of type $\br_i$ as considered in Lemma~\ref{lem:flags}.
This yields the result. 
\end{proof}

Let $\ubr=(\br_1,\ldots,\br_l)$ be a sequence of rank vectors for
$H(k)$, and let $\bm := (m_1,\ldots,m_n) = \br_1+\cdots+\br_l$. 
Set $\bd = (kc_1m_1,\ldots,kc_nm_n)$.
We consider
\begin{multline*}
\RepFlf_\ubr^{H(k)}(\bm) := \{ (M,U_\bullet) \in 
\rep_\vp^\fib(H(k),\bm) \times \Flf_\ubr^{S(k)} \mid \\
M_{ij}({_i}H_j(k)\otimes_j U_{h,j}) \subseteq U_{h,i}\text{ for all } 
(i,j)\in\Ome, 1\leq h\leq l\}
\end{multline*}
which is a subbundle of rank 
\[
\dim_K \rep_\vp^\fib(H(k),\bm) -
\left(\sum_{a<b}(\bil{\br_a,\br_b}_{S(k)}-\bil{\br_a,\br_b}_{H(k)})\right)
\]
of the trivial vector bundle
$\rep_\vp^\fib(H(k),\bm)\times\Flf_\ubr^{S(k)} \to \Flf_\ubr^{S(k)}$. 
Now we can use Corollary~\ref{cor:Sflag} and get that
$\RepFlf_\br^{H(k)}(\bm)$ is smooth and irreducible of dimension
\[
\dim_K \rep_\vp^\fib(H(k),\bm)+\sum_{a<b}\bil{\br_a,\br_b}_{H(k)}.
\] 
Let
\[
q\df \RepFlf_{\ubr}^{H(k)}(\bm)\ra \rep_\vp^\fib(H(k),\bm), \qquad 
(M, U_\bullet)\mapsto M
\]
be the restriction of the projection to the first component. Then
obviously $q^{-1}(M)=\Flf_\ubr^{H(k)}(M)$.

Finally, with the notation from Section~\ref{sec:not} we have the following 
result, whose proof is directly inspired by \cite[Theorem~5.34]{W}.

\begin{Prop} \label{prp:sm-irr}
If $\Flf_\ubr^{H(k)}$ is non-empty, then it is smooth and irreducible of dimension
\[
\sum_{a<b}\bil{\br_a,\br_b}_{H(k)}.
\] 
\end{Prop}

\begin{proof}
Since $M := M_{\bm,k}$ is rigid, the
$G(H(k),\bm)$-orbit $\cO_M$ of $M$ in $\rep_\vp^\fib(H(k),\bm)$ is open and 
dense. Thus, with $q$ defined as above and by our 
hypothesis, $q$ is dominant. By Chevalley's Theorem 
(see for example~\cite[II, Ex.~3.22]{H}) it follows, that each
irreducible component of $q^{-1}(M)=\Flf_\ubr^{H(k)}$ has at least
dimension
\begin{equation} \label{eq:dim}
\sum_{a<b}\bil{\br_a,\br_b}_{H(k)}=
\dim\RepFlf_\ubr^{H(k)}(\bm)-\dim\rep_\vp^\fib(H(k),\bm).
\end{equation}
By the same token, $q^{-1}(\cO_M)$ is dense and open in the irreducible
variety $\RepFlf_{\ubr}^{H(k)}(\bm)$. Moreover, $q^{-1}(M')\cong q^{-1}(M)$ for all
$M'\in\cO_M$. Thus, again by Chevalley's Theorem $q^{-1}(M)$ is 
equidimensional.

To prove the  last claim, namely the irreducibility 
and smoothness  of $\Flf_\ubr^{H(k)}(M) \cong \Grlf_{\ube}^{H(k,l)}(\ul{M}^{(l)})$, 
we proceed as follows:
Recall firstly, 
that we defined a rank vector
$\ube:=(\br_1,\br_1+\br_2,\ldots,\br_1+\cdots+\br_{l-1})$ 
for $H(k,l)$. 
In the following diagram we describe successively the different varieties and 
show that they are smooth and irreducible. 
For the horizontal arrows we use that
an open subset of a smooth irreducible variety is again smooth and irreducible. 
For the vertical arrows we use that in case $X\ra Y$ is a
principal $G$-bundle for a connected and smooth group $G$, then 
$X$ is smooth and irreducible if and only if $Y$ is so. 
The latter applies in particular to vector bundles.
We start in the SW corner.
\[
\xymatrix{
 &{\Hrepm(\ube,M)_{\min}}{\ar[d]^{{\text{v.b.}}}}
&{\ar@{_{(}->}[l]_{{\text{open}}}}
{\IHrepm(\ube,M)_{\min}}\ar[d]^{{G(H(k,l),\be)\text{-princ.}}}\\
{\Repm(\ube)}\ar[d]^{{\text{v.b.}}}
&{\ar@{_{(}->}[l]_<(.25){{\text{open}}}}
{\Repm(\ube,M)_{\min}}&{\Grlf_\ube^{H(k,l)}(\ul{M}^{(l)})}\\
{\Mono_{S(k)}(\ube)}
}
\]
First, let $\Hom_{S(k)}(\ube) := \{\text{pt}\}$ if $l=2$.
Otherwise, define
\[
\Hom_{S(k)}(\ube) := 
\prod_{i=1}^{l-2} \Hom_{S(k)}(S(k)^{\be_i}, S(k)^{\be_{i+1}}).
\]
Next, let $\Mono_{S(k)}(\ube) := \{\text{pt}\}$ for $l=2$, and
\begin{multline*}
\Mono_{S(k)}(\ube):=\{\ul{\nu}:=(\nu_1,\ldots,\nu_{l-2}) \in 
\Hom_{S(k)}(\ube) \mid \\
\nu_i \text{ is a monomorphism for } i=1,2,\ldots,l-2 \}.
\end{multline*}
if $l \ge 3$.
Thus, $\Mono_{S(k)}(\ube)$ is smooth and irreducible as an open subset of the 
vector space $\Hom_{S(k)}(\be)$.
Define 
\[
\rep_\vp^\fib(H(k,l),\ube) := \left(\prod_{i=1}^{l-1}\rep_\vp^\fib(H(k),\be_i)\right)\times\Hom_{S(k)}(\ube),
\]
and let
\begin{multline*}
\Repm(\ube):=\{(U_1,\ldots,U_{l-1};\ul{\nu})
\in\left(\prod_{i=1}^{l-1}\rep_\vp^\fib(H(k),\be_i)\right)\times\Mono_{S(k)}(\ube)\mid\\
\nu_i\in\Hom_{H(k)}(U_i,U_{i+1})\text{ for } i=1,2,\ldots,l-2\}.
\end{multline*}
This is, with the projection to the last component, a vector bundle over
$\Mono_{S(k)}(\ube)$. 
Thus the irreducibility and smoothness  of 
$\Repm(\ube)$ follows from the corresponding properties of $\Mono_{S(k)}(\ube)$. 
Clearly $\Repm(\ube)$ is an open subset of 
$\rep_\vp^\fib(H(k,l),\ube)$. 
The function
\[
h_M\df\Repm(\ube) \ra \NN
\]
defined by
\[
(U_\bullet,\ul{\nu})\mapsto
\dim\Hom_{H(k,l)}((U_\bullet,\ul{\nu}),\ul{M}^{(l)})
\]
is upper semicontinuous, see for example~\cite[Section~2.1]{B}.
Since $\Ext^2_{H(k,l)}(-,\ul{M}^{(l)})=0$, 
a lower bound for $h_M$ on $\Repm(\ube)$ is given by 
\[
\bil{\ube,\rkv_{H(k,l)}(\ul{M}^{(l)})}_{H(k,l)}=\bil{\be_{l-1},\rkv(M)}_{H(k)}, 
\]
see Proposition~\ref{prp:bilhkl}. 
This minimum  is achieved for any
$(U_\bullet,\ul{\nu})\in\Grlf_{\ube}^{H(k,l)}(\ul{M}^{(l)}) \neq \varnothing$ by 
Lemma~\ref{lem:hklbil2}. Thus,
\begin{multline*}
\Repm(\ube,M)_{\min} := \{ (U_\bullet,\ul{\nu}) \in \Repm(\ube) \mid\\ 
\dim \Hom_{H(k,l)}((U_\bullet,\ul{\nu}), \ul{M}^{(l)}) =
\bil{\ube,\rkv_{H(k,l)}(\ul{M}^{(l)})}_{H(k,l)} \}
\end{multline*}
is a non-empty open subset of the smooth irreducible quasi-projective variety 
$\Repm(\ube)$. 
Thus, $\Repm(\ube,M)_{\min}$ is smooth and irreducible.

Let
\begin{multline*}
\Hrepm(\ube,M)_{\min}:=\\
\{(f,(U_\bullet,\ul{\nu}))\in
\Hom_K((U_\bullet,\ul{\nu}), \ul{M}^{(l)})\times\Repm(\ube,M)_{\min}\mid\\
f\in\Hom_{H(k,l)}((U_\bullet,\ul{\nu}),\ul{M}^{(l)})\}.
\end{multline*}
Together with the projection to the second component, this is a vector
bundle over $\Repm(\ube,M)_{\min}$, see for example 
\cite[Section~2.2]{B}.
Thus, the variety  $\Hrepm(\ube,M)_{\min}$ is
smooth and irreducible because $\Repm(\ube,M)_{\min}$ is so.
Define
\[
\IHrepm(\ube,M)_{\min} := 
\{(f,(U_\bullet,\ul{\nu})\in\Hrepm(\ube,M)_{\min}\mid
f \text{ is injective}\}.
\]
By the discussion above, this is a
non-empty open subset of $\Hrepm(\ube,M)_{\min}$. 
In particular, it is smooth and irreducible because $\Hrepm(\ube,M)_{\min}$ 
is so.

Finally, we have the canonical morphism
\[
\IHrepm(\ube,M)_{\min}\ra\Grf_\ube^{H(k,l)}(\ul{M}^{(l)}), 
(f,(U_\bullet,\ul{\nu}))\mapsto\Ima(f)
\]
It is not hard to see, that this is a principal $G(H(k,l),\ube)$-bundle. 
Thus,  $\Grf_{\ube}^{H(k,l)}(\ul{M}^{(l)})\cong\Flf_\ubr^{H(k)}(M)$ is smooth and 
irreducible because $\IHrepm(\ube,M)_{\min}$ is so.
\end{proof}

Alternatively, we can prove the smoothness statement in  
Proposition~\ref{prp:sm-irr} with a short tangent space argument.
In fact, since $M$ is rigid, by 
Corollary~\ref{Cor:FRigTang},  Equation~\eqref{eq:dim} is
the dimension of the tangent 
space at each point $U_\bullet$ of $\Flf_\ubr^{H(k)}(M)$. 
This implies smoothness (see for example~\cite[III.10.0.3]{H} together 
with~\cite[Ex.~II.2.8]{H} and the basic commutative algebra 
fact~\cite[I.5.2A]{H}).

Given this, it would be nice to have a similar short argument for the
connectedness of $\Flf_\ubr^{H(k)}$, perhaps by some variant 
of~\cite[III.11.3]{H}, 
or by the first Bertini Theorem~\cite[II.6.1]{Sh}.

\subsection{A reduction map for Grassmannians}
For $k\geq 2$ let $B$  be a finite-dimensional $Z(k)$-algebra which 
is free as a $Z(k)$-module.
We consider $Z(k)$ as a subalgebra of $B$.
We set $\ov{B} := B/(\vep^{k-1}B)$, this is a 
$Z(k-1)$-algebra which is free as a $Z(k-1)$-module. 
Moreover, we set $\ov{\ov{B}}:=B/(\vep B)$. 
We have the following useful result, which is easy to prove.

\begin{Lem} \label{lem-4.1}
Let $M$ be a $B$-module which is free as a $Z(k)$-module, 
then for $1 \le j \le k-1$ multiplication by $\vep$ induces an isomorphism
of $\ov{\ov{B}}$-modules 
\[
\vep^{j-1}M/(\vep^jM)\ra \vep^jM/(\vep^{j+1}M).
\]
In particular, $M/(\vep M)\cong \vep^{k-1}M$ as $\ov{\ov{B}}$-modules.
\end{Lem}

Suppose that $e<\rk_{Z(k)}(M)$ then we denote by $\Grf^B_{Z(k),e}(M)$ the 
\emph{quiver Grassmannian} of submodules of $M$ which are free of rank $e$ as 
$Z(k)$-modules. 
It is easy to see, that $\Grf^B_{Z(k),e}(M)$ is an open subset of
the usual quiver Grassmannian $\Gr_{ke}^B(M)$ of $(k\cdot e)$-dimensional
submodules of $M$. 
In particular, $\Grf_{Z(k),e}^B(M)$ is a quasi-projective
variety. 
The reduction $M\mapsto \ov{M}$ with 
$\ov{M} := M/(\vep^{k-1}M)$ induces a
morphism of varieties 
\[
\pi\df\Grf_{Z(k),e}^B(M)\ra \Grf_{Z(k-1),e}^{\ov{B}}(\ov{M})
\]
defined by $U\mapsto \ov{U}$.

\begin{Lem}\label{lem:closed1}
With the above notation let $U \in \Grf^B_{Z(k),e}(M)$.
\begin{itemize}
 
\item[(a)]
We have
$\pi^{-1}(\pi(U))\cong
\Hom_B(U/(\vep U), (M/U)/(\vep(M/U)))$.

\item[(b)] 
Suppose there is an open neighbourhood $\cU$ of $\pi(U)$ in 
$\Grf_{Z(k-1),e}^{\ov{B}}(\ov{M})$ such that for all $U'\in\cU$ holds
\begin{multline*}
\dim_K \Hom_{\ov{B}}(U'/(\vep U'), (\ov{M}/U')/(\vep(\ov{M}/U')))\\
=\dim_K \Hom_{\ov{B}}(\pi(U)/(\vep(\pi(U))), 
(\ov{M}/\pi(U))/(\vep (\ov{M}/\pi(U))).
\end{multline*}
Then we have that $\Ima(\pi) \cap \cU$ is closed in $\cU$.

\end{itemize}
\end{Lem}

\begin{proof}
Fix a $Z(k)$-basis $\cB$ of $B$  
and $Z(k)$-basis $\{m_1,\ldots,m_l\}$ of $M$, which we identify with
the standard basis of $Z(k)^l$. 
We may suppose that $U$ is the $Z(k)$-span of $\{m_1,\ldots, m_e\}$.
Thus, the action of $B$ on $M$ is described by the matrices
\[
M(b)=\sum_{j=0}^{k-1} 
\begin{pmatrix} U_j(b)&E_j(b)\\ 0&F_j(b)\end{pmatrix} \vep^j \in
Z(k)^{m\times m}
\]
\[
\text{ with } U_j(b)\in K^{e\times e},\ E_j(b)\in K^{e\times(m-e)},\  
F_j(b)\in K^{(m-e)\times(m-e)}
\]
for all  $b\in\cB$ and  $j=1,2,\ldots,k-1$.

(a) With this setup the free $Z(k)$-submodule $\tilde{U}$ of 
${_{Z(k)}M}=Z(k)^m$ with $\tilde{U}/\vep^{k-1}\tilde{U}=U/\vep^{k-1} U$
are precisely those which are  spanned by the columns of a matrix of 
the form
\[
\begin{pmatrix} \bbo_e\\ \vep^{k-1} S\end{pmatrix}\in Z(k)^{m\times e}
\]
for some $S\in K^{(m-e)\times e}$. Now, such a subspace $\tilde{U}_S$
is a $B$-submodule of $M$ if and only if 
for all $b \in \cB$ the left lower $(m-e)\times e$-block of
\begin{multline*}
\begin{pmatrix}\bbo_e & 0\\ \vep^{k-1}S& \bbo_{m-e} \end{pmatrix}^{-1}
M(b) \begin{pmatrix}\bbo_e & 0\\ \vep^{k-1}S& \bbo_{m-e}\end{pmatrix}\\
=\begin{pmatrix}\bbo_e & 0\\-\vep^{k-1}S& \bbo_{m-e} \end{pmatrix} M(b)
\begin{pmatrix}\bbo_e & 0\\ \vep^{k-1}S& \bbo_{m-e} \end{pmatrix}
\quad\in Z(k)^{m\times m}
\end{multline*}
vanishes.
In other words, if $F_0(b)\cdot S-S\cdot U_0(b)=0$ for all $b \in \cB$. 
This means that by Lemma~\ref{lem-4.1}, $\pi^{-1}(\pi(U))$ can be identified 
with 
$\Hom_{B}(U/\vep U, F/\vep F)$ for $F=M/U$.

(b) We may assume that $\cU$ is contained in the standard open neighbourhood of
$\pi(U)$ with respect to our chosen basis $\{\ov{m}_i\mid i=1,\ldots, l\}$
for $\ov{m}_i:= m_i+\vep^{k-1}M$. 
The elements of this neighbourhood are submodules 
$U' = U_{\ov{S}}$ which are 
$Z(k-1)$-spanned by the columns of a matrix of the form 
\[
\begin{pmatrix} \bbo_e\\ \ov{S}\end{pmatrix} \text{ with }
\ov{S}=\sum_{j=0}^{k-2} S_j\vep^j\in Z(k-1)^{e\times (m-e)}
\]
subject to the condition that the left lower $(m-e)\times e$-block of the
matrices
\[
\begin{pmatrix}\bbo_e & 0\\-\ov{S}& \bbo_{m-e} \end{pmatrix} \ov{M}(b)
\begin{pmatrix}\bbo_e & 0\\ \ov{S} & \bbo_{m-e} \end{pmatrix}
\quad\in Z(k)^{m\times m}
\]
vanishes for all $b\in\cB$. 
More explicitly, this means
\[
0=\ov{F}(b)\cdot\ov{S}-\ov{S}\cdot\ov{U}(b)-\ov{S}\cdot\ov{E}(b)\cdot\ov{S}
\in Z(k-1)^{(m-e)\times e} \text{ for all } b\in\cB.
\]
Now, $U_{\ov{S}}$ is in the image of $\pi$ if and only if there exists 
$S_{k-1}\in K^{(m-e)\times e}$ such that for 
\[
\tilde{S} := \sum_{i=0}^{k-1} S_i\vep^i\in Z(k)^{(m-e)\times e}
\]
the $Z(k)$-span of the columns of the matrix
\[
\begin{pmatrix} \bbo_e\\\tilde{S}\end{pmatrix}\in Z(k)^{m\times e}
\]
is a $B$-submodule of $M$. This is the case if and only if the following,
possibly non-homogeneous system of linear equations in the components of $S_{k-1}$
has a solution:
\begin{multline} \label{eq:impi}
(F_0(b)- S_0\cdot E_0(b)) \cdot S_{k-1}- S_{k-1}\cdot (U_0(b)+E_0(b)\cdot S_0)\\
= \sum_{0\leq i,j\leq k-2} S_i\cdot E_{k-1-i-j}(b)\cdot S_j 
\text{ for all } b\in\cB, 
\end{multline}
where we have set $E_h(b)=0$ for $h<0$. 
Now observe that the family of matrices
$(F_0(b)- S_0\cdot E_0(b))_{b\in\cB}$ describe the (reduced) factor module 
$(M/U_{\ov{S}})/(\vep (M/U_{\ov{S}}))$, and the family
$(U_0(b)+E_0(b)\cdot S_0)_{b\in\cB}$ describes the (reduced) submodule 
$U_{\ov{S}}/(\vep U_{\ov{S}})$. 
Thus, by hypothesis the system of linear equations~\eqref{eq:impi} 
has constant rank for all points $U_{\ov{S}} \in \cU$. It follows, that the subset
where this system has a solution, is closed.
\end{proof}

Example~\ref{ex:a2-2} illustrates the calculations in the above proof.

\subsection{Closed image}
We show that Lemma~\ref{lem:closed1}(b) can be used to conclude that the
image of the reduction morphism is closed.

\begin{Lem}\label{prp:closed}
Let $k\geq 2$, and let $M$ be a rigid locally free $H(k)$-module. For a sequence
of rank vectors $\ubr=(\br_1,\ldots,\br_l)$ with 
$\br_1+\cdots+\br_l=\rkv_{H(k)}(M)$ we consider the natural reduction morphism
\[
\pi_k\df \Flf_{\ubr}^{H(k)}(M) \ra \Flf_{\ubr}^{H(k-1)}(\ov{M})
\] 
defined by $U_\bullet \mapsto U_\bullet/\vep^{k-1}U_\bullet$.
Then the image
of $\pi_k$ is closed in $\Flf_{\ubr}^{H(k-1)}(\ov{M})$.
\end{Lem}

\begin{proof}
Writing $\ube := (\br_1,\br_1+\br_2,\ldots,\br_1+\cdots+\br_{l-1})$, 
it is sufficient
to show the following equivalent claim: the corresponding reduction map
\[
\pi'_k\df\Grlf_{\ube}^{H(k,l)}(\ul{M}^{(l)})\ra
\Grlf_{\ube}^{H(k-1,l)}(\ov{\ul{M}}^{(l)})
\]
has closed image, see Proposition~\ref{prp:fl-gr}.

To this end we observe that $B:=H(k,l)$ is a $Z(k)$-algebra which
is free as $Z(k)$-module, and 
$\ov{B}:=H(k-1,l)\cong H(k,l)/(\vep^{k-1} H(k,l))$. 
Now, we show that $\Grlf_\ube^{\ov{B}}(\ov{\ul{M}}^{(l)})$ is an open subset
of $\Grf_{Z(k-1),e}^{\ov{B}}(\ov{\ul{M}}^{(l)})$ for an adequate $e \in \N$.
This subset fulfills the hypothesis
of Lemma~\ref{lem:closed1}(b) as we shall now see.
Indeed, $\Grlf^{\ov{B}}_\ube(\ov{\ul{M}}^{(l)})$
is an open subset of $\Grf^{\ov{B}}_{Z(k-1),e}(\ov{\ul{M}}^{(l)})$, 
since being locally free is an open condition by the argument at the
end of the proof of Proposition~\ref{prp:fl-gr}.  
Moreover, each $U \in \Grlf_\ube^{\ov{B}}(\ov{\ul{M}}^{(l)})$ yields a short exact
sequence of locally free $H(1,l)$-modules
\[
0\ra U/(\vep U) \ra \ov{\ul{M}}^{(l)}/(\vep \ov{\ul{M}}^{(l)}) \ra 
(\ov{\ul{M}}^{(l)}/U)/(\vep\ov{\ul{M}}^{(l)}/U)\ra 0.  
\]
Now, by Proposition~\ref{prop-red} and since $M$ is rigid, $M/\vep M$ is a 
rigid $H(1)$-module. Thus, since
\[ 
\ov{\ul{M}}^{(l)}/(\vep\ov{\ul{M}}^{(l)})\cong \ul{M/\vep M}^{(l)}
\]
as $H(1,l)$-module, we can apply Lemma~\ref{lem:hklbil2} to the
above exact sequence to see that 
\[
U\mapsto\dim\Hom_{H(1,l)}(U/(\vep U),(\ov{\ul{M}}^{(l)}/U)/(\vep\ov{\ul{M}}^{(l)}/U))
\]
is a constant function on $\Grlf_\ube^{\ov{B}}(\ov{\ul{M}}^{(l)})$.
It then follows from Lemma~\ref{lem:closed1} (with $M$ replaced by $\ul{M}^{(l)}$
and $\cU = \Grlf_\ube^{\ov{B}}(\ov{\ul{M}}^{(l)})$) that 
$\Ima(\pi) \cap \cU$ is closed in $\cU$.

Finally, for rank reasons, no locally free submodule 
$U'\in\Grlf_\ube^{\ov{B}}(\ov{\ul{M}}^{(l)})$ can be the reduction of 
some $U\in\Grf_{Z(k),e}^{B}(\ul{M}^{(l)})$ which is not itself a locally
free $B$-module. 
\end{proof}

\subsection{Conclusion of the proof of Theorem~\ref{thm-mainthm2}} \label{sect-concl}
We know by Proposition~\ref{prop-red}(b) that
for any $k\geq 2$ the  
quiver flag variety $\Flf^{H(k)}_\ubr$ is
non-empty if and only if $\Flf_\ubr^{H(1)}$ is non-empty.
Using this and 
Proposition~\ref{prp:sm-irr} it remains
only to show that $\pi_k$ is surjective with all fibers being isomorphic
to an affine space of dimension $d(\ubr)$. 
Indeed, by Lemma~\ref{lem:closed1}(a) and Lemma~\ref{lem:hklbil2} we
have that
$\pi_k^{-1}(\pi_k(U_\bullet))$ is 
an affine space of dimension $d(\ubr)$ for all $U_\bullet\in\Flf_\ubr^{H(k)}$. 
By Chevalley's theorem
and Lemma~\ref{prp:closed}, $\Ima(\pi_k)$ is a closed subset of
dimension $\dim\Flf_\ubr^{H(k)}-d(\ubr)$ in $\Flf_\ubr^{H(k-1)}$. 
Now, by Proposition~\ref{prp:sm-irr} we have
\[\dim\Flf_\ubr^{H(k-1)} =
\dim\Flf_\ubr^{H(k)} - d(\ubr).
\]
Since, again by Proposition~\ref{prp:sm-irr}, $\Flf_\ubr^{H(k-1)}$ is
irreducible, and $\Ima(\pi_k)$ is closed we conclude that $\pi_k$ is
indeed surjective.
This finishes the proof of Theorem~\ref{thm-mainthm2}

Note that the surjectivity of the morphism 
$\pi_k\df \Flf_{\ubr}^{H(k)}(M) \ra \Flf_{\ubr}^{H(k-1)}(\ov{M})$
does not follow from Proposition~\ref{prop-red}(b).
Using the notation from Proposition~\ref{prop-red}(b), for a fixed
locally free module $X_l$ we might need different modules $Y_l$
to lift all flags of locally free submodules of $X_l$.

\subsection{Example}\label{ex:a2-2}
We study the same family of algebras $H(k) = H(C,kD,\Omega)$ as in 
Example~\ref{ex:a2-1}. 
Thus $C$ is symmetric and $D$ is the identity matrix.
It follows that an $H(k)$-module is locally free if and only if it is free as a 
$Z(k)$-module.

We consider a locally free $H(k)$-module $N^{(k)}$ with
$\rkv_{H(k)}(N^{(k)})=(2,2)$. 
It is determined by the $Z(k)$-linear map 
\[
N_{12}^{(k)}\df {_1}H_2(k) \otimes_{H_2(k)} N_2^{(k)} \to N_1^{(k)}. 
\]
We choose this map such that with respect to the
standard bases of $N_1^{(k)} := H_1(k)^2$ and $N_2^{(k)} := H_2(k)^2$ it is
represented by the matrix $\left(\bsm 0&1\\0&0\esm\right)$.
(Note that ${_1}H_2(k) \cong H_2(k)$ as $H_2(k)$-modules in this case, so that
${_1}H_2(k) \otimes_{H_2(k)} N_2^{(k)} \cong N_2^{(k)}$.)
Thus the three isomorphism classes
of rigid locally free indecomposable $H(k)$-modules appear precisely
with multiplicity one.
Namely we have
\[
N^{(k)}\cong E_1^{(k)}\oplus E_2^{(k)}\oplus P_2^{(k)}.
\]
In particular, $N^{(k)}$ itself is not rigid. 
It is known 
(see for example~\cite[Example~6.3]{W}) that
\[
\Gr_{(1,1)}^{H(1)}(N^{(1)})\subset \PP^1(K)\times\PP^1(K)
\] 
consists of two 
irreducible components, each of them isomorphic to $\PP^1(K)$, 
which intersect in exactly one point $U=([1:0],[1:0])$. 
A similar example was studied in \cite[Example~3.7]{DWZ}.
Note that $U\cong S_1\oplus S_2\cong N^{(1)}/U$ and consequently
we find for the tangent space 
\[
T_U\Gr_{1,1}^{H(1)}(N^{(1)})\cong\Hom_{H(1)}(U, N^{(1)}/U)\cong K^2,
\]
see~\cite[Lemma~3.2]{Sc}.

Now consider $\PP^1(Z(k)):=\Grf_{Z(k),1}^{H_i(k)}(H_i(k)^2)$, 
i.e. the space of 
rank $1$ free submodules of $H_i(k)^2\cong Z(k)^2$. We use a similar
notation for the points of this space as the usual notation of points
in $\PP^1(K)$, namely $[a:b]$ with $a,b\in Z(k)$ and at least one of
$a$ and $b$ must be invertible in $Z(k)$; moreover $[a:b]\sim [a':b']$
if and only if there exists an invertible element $u\in Z(k)^\times$ such that
$(a',b')=(ua,ub)$.  
It follows that $\PP^1(Z(k))$ is the disjoint union: 
\[
\PP^1(Z(k)) = \{[1:a]\mid a\in Z(k)\} \stackrel{\mathbf{\cdot}}{\cup} 
\{[\vep a':1]\mid a'\in Z(k)\}.
\]
With this notation we have
\[
\Grlf_{(1,1)}^{H(k)}(N^{(k)}) 
\subset
\PP^1(Z(k))\times\PP^1(Z(k)).
\]
Consider now elements of the form
\[
U_{a,b} = ([1:a],[1:b])\in\PP^1(Z(k))\times\PP^1(Z(k))
\text{ with } a,b\in Z(k).
\]
It is easy to see that
\[
\cU^{(k)} := \{ U_{a,b} \in \PP^1(Z(k))\times\PP^1(Z(k))
\mid 0=a\cdot b\in Z(k)\}
\]
forms a dense open subset of $\Grlf_{(1,1)}^{H(k)}(N^{(k)})$. 
We leave it as an
exercise to show that $\cU^{(k)}$ has precisely $k+1$ irreducible 
components, each of them isomorphic to an affine space of dimension $k$.

Now, for $U = ([1:0],[1:0])$ as above, we have that 
$\pi_2^{-1}(U)=\{U_{\vep a,\vep b}\mid a,b\in K\}$ is precisely one of
the three irreducible components of $\cU^{(2)}$. Clearly,
$\pi^{-1}_3\left(\pi_2^{-1}(U)\right)\subset\cU^{(3)}$. Due to the defining
equations of $\cU^{(3)}$ we see that 
\[
\pi^{-1}_3(U_{\vep a,\vep b})=\begin{cases}
\{U_{\vep^2 a',\vep b+\vep^2 b'}\mid a',b'\in K\}& \text{ if } a =0,\\
\{U_{\vep a+\vep^2 a',\vep^2 b'}\mid a',b'\in K\}& \text{ if } b =0,\\
\varnothing &\text{ if } ab\neq 0.
\end{cases}
\]
In particular, in this situation
\[ 
\pi_3\df \Grlf_{(1,1)}^{H(3)}(N^{(3)})\ra \Grlf_{(1,1)}^{H(2)}(N^{(2)})
\]
is not surjective, 
and $\Grlf^{H(k)}_{(1,1)}(N^{(k)})$ is neither
smooth nor irreducible.


\section{Conjectures}


\subsection{Irreducibility}
Let $H = H(C,D,\Omega)$.
In general it is easy to find examples of dimension vectors $\bd$ such that the 
variety $\rep(H,\bd)$ is not irreducible.

\begin{Conj}\label{conj-irreducible}
Assume that $\bd$ is a $D$-divisible dimension vector for $H$.
Then $\rep(H,\bd)$ is irreducible.
In other words,
$\repvp(H,\br)$ is dense in $\rep(H,\bd)$.
\end{Conj}

\subsection{Number of parameters}
For $H(k) = H(C,kD,\Omega)$ let $\br$ be a rank vector.
The \emph{number of parameters} of $\br$ is defined as
\[
\mu_{H(k)}(\br) := \dim \repvp(H(k),\br) - \max \{ \dim \cO_M \mid 
M \in \repvp(H(k),\br) \}.
\]
For example, if there exists some rigid $M \in \repvp(H(k),\br)$, then
$\mu_{H(k)}(\br) = 0$, since $\cO_M$ is open in this case.

\begin{Conj}\label{conj-parameters}
For $k \ge 2$ and all rank vectors $\br$ we have
$\mu_{H(k)}(\br) = k \cdot \mu_{H(1)}(\br)$.
\end{Conj}

\subsection{A generalization of Kac's Theorem}
Again let $H = H(C,D,\Omega)$.
Recall from \cite{GLS1} that an $H$-module $M$ is $\tau$-\emph{locally free} 
provided $\tau_H^k(M)$ is locally free for all $k \in \Z$.
Here $\tau_H(-)$ denotes the Auslander-Reiten translation of $H$.

\begin{Conj}\label{conj-kac}
There is a bijection between the set of positive roots of the Kac-Moody Lie 
algebra $\g(C)$ associated with $C$ 
and the set of rank vectors of indecomposable $\tau$-locally free $H$-modules.
\end{Conj}

For $C$ symmetric and $D$ the identity matrix, Conjecture~\ref{conj-kac}
is true and was proved by Kac \cite{K1,K2}.
For $C$ of Dynkin type and $D$ arbitrary, 
Conjecture~\ref{conj-kac} is true, see \cite{GLS1}.

\subsection{Schur roots}
For $H = H(C,D,\Omega)$ let $\br$ be an $H$-Schur root.
Thus $\indvp(H,\br)$ is dense in $\rep_\vp(H,\br)$.
We conjecture that for a dense open subset $U \subseteq \indvp(H,\br)$
there is some $1 \le i \le n$ such that $\End_H(M) \cong H_i$
for all $M \in U$.
This would be a generalization of \cite[Proposition~1(a)]{K2}.

\subsection{Schofield's algorithm}
Let $H = H(C,D,\Omega)$.
We conjecture that Schofield's algorithm \cite{Sc} (see also \cite{DW}) 
for the computation of canonical decompositions of dimension vectors for path 
algebras can be generalized to
an algorithm which computes the $H$-canonical decomposition 
of a given rank vector $\br$.

\bigskip
{\parindent0cm \bf Acknowledgements.}\\
The first author acknowledges financial support from UNAM-PAPIIT grant 
IN108114 and Conacyt Grant 239255.
The third author thanks the SFB/Transregio TR 45 for financial support.


\end{document}